\documentclass[11pt,reqno]{amsart}
\usepackage[normalem]{ulem}
\PassOptionsToPackage{table}{xcolor}
\usepackage{stmaryrd}
\expandafter\def\csname opt@stmaryrd.sty\endcsname
{only,shortleftarrow,shortrightarrow}
\usepackage{extpfeil}
\usepackage{amsmath,amssymb,amsthm}
\usepackage{aliascnt}
\usepackage{varioref}
\usepackage{hyperref}
\usepackage{thmtools}
\usepackage[capitalize,nameinlink,noabbrev]{cleveref}
\hypersetup{colorlinks=true,urlcolor=blue,citecolor=blue,linkcolor=blue}
\usepackage{courier}
\usepackage{tikz}
\usepackage{tikz-cd}
\usetikzlibrary{calc,matrix,arrows,decorations.markings}
\usepackage{array}
\usepackage{color}
\usepackage{enumerate}
\usepackage{nicefrac}
\usepackage{listings}
\usepackage{seqsplit}
\usepackage{longtable}
\usepackage{colonequals}
\usepackage{collcell}
\usepackage{booktabs}
\usepackage{caption}
\usepackage{comment}
\usepackage{array, multirow}
\newcommand{\F}{\mathbb{F}}
\newcommand{\Q}{\mathbb{Q}}
\newcommand{\Qbar}{\overline{\Q}}
\newcommand{\R}{\mathbb{R}}
\newcommand{\C}{\mathbb{C}}

\newcommand{\Z}{\mathbb{Z}}

\newcommand{\Aut}{{\rm Aut}}

\renewcommand{\P}{\mathbb{P}}

\newcommand{\Spec}{\operatorname{Spec}}

\newcommand{\Jac}{\operatorname{Jac}}

\newcommand{\new}{\operatorname{new}}

\DeclareFontFamily{U}{wncy}{}
    \DeclareFontShape{U}{wncy}{m}{n}{<->wncyr10}{}
    \DeclareSymbolFont{mcy}{U}{wncy}{m}{n}
    \DeclareMathSymbol{\Sh}{\mathord}{mcy}{"58}

\newcolumntype{M}{>{\collectcell\ensuremath}l<{\endcollectcell}} 
\newcolumntype{R}{>{\collectcell\ensuremath}c<{\endcollectcell}}

\newtheorem{theorem}{Theorem}[section]

\newtheorem{lemma}[theorem]{Lemma}
\newtheorem{proposition}[theorem]{Proposition}
\newtheorem{corollary}[theorem]{Corollary}
\newtheorem{algorithm}[theorem]{Algorithm}

\theoremstyle{definition}
\newtheorem{definition}[theorem]{Definition}

\newtheorem{remark}[theorem]{Remark}

\newcommand{\Claudio}[1]{{\color{violet}{;-) #1}}}
\newcommand{\Pietro}[1]{{\color{orange}{??Pietro: #1}}} 
\newcommand{\freddy}[1]{{\color{purple}{$\spadesuit$ Freddy: #1}}}

\title{Point counts, automorphisms, and gonalities of Shimura curves}

\author{\sc Pietro Mercuri}
\address{Pietro Mercuri \\
University of Palermo\\
Italy}
\urladdr{https://sites.google.com/view/mercuriptr/home}
\email{mercuri.ptr@gmail.com}

\author{\sc Oana Padurariu}
\address{Oana Padurariu \\
Max-Planck-Institut für Mathematik Bonn\\
Germany}
\urladdr{https://sites.google.com/view/oanapadurariu/home}
\email{oana.padurariu11@gmail.com}

\author{\sc Frederick Saia}
\address{Frederick Saia \\
Lafayette College\\
USA}
\urladdr{https://fsaia.github.io/site/}
\email{saiaf@lafayette.edu}

\author{\sc Claudio Stirpe}
\address{Claudio Stirpe\\
Convitto Regina Margherita, Anagni\\
Italy
}
\email{clast@inwind.it}

\begin{document}


\begin{abstract}
We implement an algorithm to compute the number of points over finite fields for the Shimura curves $X_0^D(N)$ over $\mathbb{Q}$ and their Atkin--Lehner quotients. Our computations identify $116$ such quotients over finite fields (out of $783514$ tested) that attain a number of rational points exceeding that of any previously known curve of the same genus over the same finite field. To illustrate the utility of our point counts algorithm in addressing arithmetic questions, we prove that all automorphisms are Atkin--Lehner for $9288$ of the $10609$ curves $X_0^D(N)$ of genus $g > 2$ with $D$ the discriminant of an indefinite quaternion algebra over $\mathbb{Q}$, $N$ a squarefree positive integer coprime to $D$, and $DN\leq 10000$, and we determine all tetragonal and geometrically tetragonal curves $X_0^D(N)$ up to a small number of possible exceptions. 
\end{abstract}

\setcounter{tocdepth}{1}
\maketitle


\section{Introduction}\label{section: introduction}

In \cite{DLMS}, Dose--Lido--Mercuri--Stirpe provide an algorithm for computing the number of points on $X$ over a finite field~$\F_q$ for $X$ a member of a wide family of modular curves and their quotients by Atkin--Lehner involutions. A main focus in the referenced article is on finding explicit curves which achieve the largest known point counts among curves of a fixed genus~$g$ over a finite field~$\F_q$ of fixed cardinality. Indeed, the referenced authors found among the studied families of modular curves many new examples achieving the largest known point counts over small finite fields of characteristic up to $19$ at the time of publication of their work based on data available at \cite{ManyPoints} (see \cite[\S 6]{DLMS}). Their work is extended in \cite{DLMS26} to a larger family including all the possible Atkin--Lehner quotients.

In the current work, we implement a similar algorithm for the Shimura curves $X_0^D(N)$ over~$\Q$ and their Atkin--Lehner quotients, parameterizing abelian surfaces with level structure and with potential quaternionic multiplication by the indefinite quaternion algebra over~$\Q$ of discriminant~$D$. The case $D=1$ recovers the modular curve setting of \cite{DLMS}, so we take $D>1$ in this work. Our main algorithm is achieved using an isogeny due to Ribet (see \cref{Ribet_isog} below) in combination with the methods of \cite{DLMS}, and it has the following form.

\begin{algorithm}\label{algorithm: introduction_version}
\phantom{space} \vspace{0.5em}

\textnormal{\textbf{Input}}: A tuple $(D,N,W,p,r)$, where:
\begin{itemize}
    \item $D>1$ is the discriminant of an indefinite quaternion algebra over~$\Q$ (equivalently, a product of an even number of distinct rational primes);
    \item $N$ is a positive integer coprime to $D$;
    \item $W$ is a subgroup of the group $W_0(D,N) \leq \Aut(X_0^D(N))$ 
    of Atkin--Lehner involutions on $X_0^D(N)$ (defined in \cref{section: AL}), given by a list $[m_1, \ldots, m_n]$ of Hall divisors of $DN$ so that $\{w_{m_1}, \ldots, w_{m_n}\}$ is a generating set for $W$;
    \item $p \nmid DN$ is a rational prime;
    \item $r$ is a positive integer.
\end{itemize}

\textnormal{\textbf{Output}}: The number $\#\left(X_0^D(N)/W\right)\left(\F_{p^r}\right)$ of $\F_{p^r}$-rational points on the base change of the curve $X_0^D(N)/W$ to the finite field $\F_{p^r}$. 
\end{algorithm}

\begin{remark}
As \cref{algorithm: introduction_version} is described above, it includes computations of eigenvalues for the action of Hecke operators on a suitable basis of cusp forms of weight~$2$ associated to isogeny factors of the Jacobian of the modular curve $X_0(DN)$. In practice, for the purpose of speed, one can take such data as input. For computations described in this paper, we use such data obtained from \cite{LMFDB} (see \cref{section: point_counts} for further discussion). 
\end{remark}

Among the family of curves studied, we find $116$ curves achieving, to our knowledge and based on current information on \cite{ManyPoints}, record-breaking point counts beyond those observed in \cite{DLMS} and \cite{DLMS26} (see \cref{tabrecord}). Additionally, we find $898$ curves that are maximal (i.e., they attain the greatest possible point count) 
for given pairs $(\text{genus}, \text{ finite field})$. In these cases, there were already known curves attaining these number of points but it was not previously known that the counts were realized by Atkin--Lehner quotients of Shimura curves. Furthermore, we find (at least) $4$ previously unknown isomorphism classes of maximal curves. 

Our implementation of \cref{algorithm: introduction_version} is also strongly motivated by potential applications to arithmetic properties of the curves $X_0^D(N)$ and their quotients. An important detail about this algorithm is that it does not require defining equations for the studied curves, as is true for the point count algorithm for modular curves of \cite{DLMS} and \cite{DLMS26}. Our knowledge of explicit equations for Shimura curves in the case $D>1$ is quite limited compared to that for modular curves, due in part to the absence of cuspidal points on these Shimura curves.

As a first example, we consider the automorphism groups of these curves over $\Q$. Restricting to squarefree level $N$, it is expected that $\Aut(X_0^D(N))$ consists only of Atkin--Lehner involutions for all but finitely many pairs $(D,N)$. A result of Gonz\'alez \cite{Gonzalez17} (\cref{lemma: Gonzalez_no_involutions} below) provides a restriction on the automorphism group of a curve over $\Q$ based on its point counts over finite fields, and using this restriction and other restrictions from work of Kontogeorgis--Rotger \cite{KR08} (\cref{lemma: KR_all_atkin_lehner} below) we implement algorithmic checks of sufficient conditions for $X_0^D(N)$ and its quotients to have no non-Atkin--Lehner involutions. We reach the following computational result (which appears in \cref{section: automorphisms} as \cref{corollary: all_atkin_lehner}). 

\begin{theorem}\label{theorem: all_atkin_lehner_introduction}
Let $D>1$ be the discriminant of an indefinite quaternion algebra over~$\Q$, let $N$ be a positive integer coprime to $D$ such that
\begin{itemize}
    \item $X_0^D(N)$ has genus $g > 2$ and 
    \item $DN \leq 10000$,
\end{itemize}
and let $W$ be a subgroup of Atkin--Lehner involutions of $\textnormal{Aut}(X_0^D(N))$.
\begin{enumerate}
    \item If $(D,N)$ does not appear among the $12$ pairs in \cref{table: unknown_automorphism_group}, then we have that $\Aut(X_0^D(N)) = W_0(D,N)$.
    \item If $(D,N)$ does not appear among the $1321$ pairs in \cref{table: unknown_automorphism_group_stars}, then $\Aut(X_0^D(N)/W)$ consists only of Atkin--Lehner involutions. Hence
\[ \Aut(X_0^D(N)/W) \cong \left(\Z/2\Z\right)^{\omega(DN) - \textnormal{ord}_2\left(|W|\right)} , \]
where $\omega(n)$ denotes the number of distinct prime divisors of an integer $n$. 
\end{enumerate}
\end{theorem}

\begin{remark}
There is a total of $10609$ coprime pairs $(D,N)$ with $D>1$ satisfying the bulleted items in \cref{theorem: all_atkin_lehner_introduction}.
\end{remark}

For a curve over~$\Q$, point counts over finite fields also supply lower bounds on gonalities over~$\Q$ (see \cref{lemma: gonality_bound_by_point_counts}). Explicit point count computations provide improvements in many cases over known bounds, such as Abramovich's bound in terms of the genus \cite{Abramovich96} (\cref{theorem: Abr} below), while other techniques from the modular curve setting (see, for example, \cite{DvH}) do not transfer over due to lack of cuspidal points.

Gonality bounds and computations are important in applications to low-degree points on modular and Shimura curves, and we expect \cref{algorithm: introduction_version} to be integral in advancing our knowledge in this area. Shimura curves lack real points \cite[Theorem 0]{Sh75} and hence lack points of any odd degree over~$\Q$. The curves $X_0^D(N)$ with $D>1$ with infinitely many degree~$2$ points were determined by Padurariu--Saia \cite[Theorem 6.3]{PS25}, building on work of Ogg \cite{Ogg83}, Guo--Yang \cite{GY17}, and Rotger \cite{Rotger02}. To determine those curves with infinitely many degree~$4$ points, a natural first question is to ask which curves $X_0^D(N)$ have gonality~$4$ over~$\Q$, that is, which admit a degree $4$ map to the projective line over~$\Q$. Using our point counts in combination with other methods, we come close to this determination. 

\begin{theorem}\label{theorem: tetragonal_introduction}
Let $D>1$ be the discriminant of an indefinite quaternion algebra over~$\mathbb{Q}$ and let $N$ be a positive integer which is relatively prime to~$D$. 
\begin{enumerate}
    \item If the Shimura curve $X_0^D(N)$ is geometrically tetragonal and it is not among the $161$ geometrically tetragonal curves $X_0^D(N)$ with $(D,N)$ listed in \cref{table: tetragonal_AL2}, \cref{table: tetragonal_AL4}, and \cref{table: geom_tetragonal_by_Polizzi}, then $(D,N)$ must be among the $29$ pairs listed in \cref{table: remaining_geom_tetragonal_candidates}. 
    \item If the Shimura curve $X_0^D(N)$ is tetragonal over~$\Q$ and it is not among the $141$ tetragonal curves $X_0^D(N)$ with $(D,N)$ listed in Part~(2) of \cref{prop: tetragonal_sieving} and in \cref{table: tetragonal_by_CM}, then $(D,N)$ must be among the $32$ pairs listed in \cref{table: remaining_tetragonal_candidates}. 
\end{enumerate}
\end{theorem}

This work is organized as follows. In \cref{section: shimura_curves}, we provide useful background and known results on Shimura curves and their automorphism groups. We do the same regarding gonalities of curves in \cref{section: gonalities}. We then outline \cref{algorithm: introduction_version} in \cref{section: point_counts}. The remaining sections are devoted to applications: \cref{section: automorphisms} and \cref{section: tetragonal} feature our applications to automorphism groups of Shimura curves and tetragonal Shimura curves, respectively, and we study curves achieving largest known point counts in \cref{section: maximal_curves}.

All computations described in this paper were performed using the Magma computer algebra system \cite{Magma}. All related code can be found in the GitHub repository \cite{MPSSRep}.


\section*{Acknowledgments}

We thank John Voight for helpful discussion and for the initial suggestion of the project. We thank Valerio Dose and Guido Lido for their contribution to the code for the point counts over finite fields. P.M.\@ is supported by the ``National Group for Algebraic and Geometric Structures, and their Applications" (GNSAGA - INdAM). F.S.\@ gratefully acknowledges support provided by an AMS-Simons Travel Grant.


\section{Shimura curves}\label{section: shimura_curves}

A $\emph{quaternion algebra}$ over a field $F$ is a central simple algebra of dimension $4$ over $F$. We call a quaternion algebra $B$ over the field $\Q$ of rational numbers \emph{indefinite} if it is split at the single real place of $\Q$, i.e., if
\[ B \otimes_\Q \R \cong M_2(\R). \]
A quaternion algebra $B$ over $\Q$ is determined up to isomorphism by its discriminant~$D$, that is the product of the rational primes at which $B$ is ramified. The indefinite quaternion discriminants over $\Q$ are exactly those integers which are products of an even number of distinct rational primes. The case $D=1$, for example, is that of the split quaternion algebra $M_2(\Q)$. 

Let $D$ be the discriminant of an indefinite quaternion algebra $B$ over $\mathbb{Q}$, let $N$ be a positive integer which is relatively prime to $D$, and let $\mathcal{O}_N$ be an Eichler order of level~$N$ in $B$ (as defined in \cite[\S 23.1.3]{Voight21}).
Consider, over the algebraic closure $\overline{\Q}$ of $\Q$, the moduli problem of parameterizing pairs $(A,\iota)$, where:
\begin{itemize}
    \item $A$ is an abelian surface, and
    \item $\iota\colon \mathcal{O}_N \hookrightarrow \textnormal{End}(A)$ is an embedding.
\end{itemize}
A theorem of Shimura \cite[Main Theorem 1]{Sh67} states that this problem is represented by a scheme which admits a canonical model over $\mathbb{Q}$, and we let $X_0^D(N)$ denote this coarse moduli scheme over $\mathbb{Q}$. The case $D=1$ of $X_0^{1}(N)$ recovers the modular curve setting, parameterizing elliptic curves with torsion level structure. When $D>1$, these curves lack the rational cusps the modular curves $X_0(N)$ possess and, moreover, they have no real points \cite[Theorem 0]{Sh75}: if $D>1$, then $X_0^D(N)(\mathbb{R}) = \varnothing$.

By the work of Drinfeld, for which we refer to \cite{BC91,Buzzard}, there exists a flat proper model of $X_0^D(N)$ over $\Z$, extending the moduli interpretation, which is smooth over $\Z\left[\tfrac{1}{DN}\right]$. For this work, it is important to note that $X_0^D(N)$ has good reduction at primes $p \nmid DN$.

For $N$ a positive integer, we have a natural modular covering map $X_0^D(N) \to X_0^D(1)$ with degree $\psi(N)$ given by the Dedekind Psi function:
\[ \psi(N) = N \prod_{\substack{p \mid N \\ \text{prime}}} \left(1+ \frac{1}{p}\right). \]
Moreover, these natural maps are compatible: for $N_1 \mid N_2$ there is a natural map $X_0^D(N_2) \to X_0^D(N_1)$ of degree $\psi(N_2)/\psi(N_1)$. 
We have the following formula for the genus of the curve $X_0^D(N)$ (see, for example, \cite[Thm. 39.4.20]{Voight21}).
\begin{proposition}\label{prop: genus_formula}
For $D>1$ the discriminant of an indefinite quaternion algebra over~$\Q$, for $N$ a positive integer coprime to $D$, and for $k \in \{3,4\}$, define
\[
e_k(D,N) \colonequals \prod_{\substack{p|D \\ \text{prime}}} \left(1 - \left(\frac{-k}{p}\right)\right) \prod_{\substack{\ell|N \\ \text{prime}}} \left(1 + \left(\frac{-k}{\ell}\right)\right) \prod_{\substack{\ell^2|N, \\ \ell \text{ prime}}}\left(\frac{-k}{\ell}\right) .
\]
The genus of $X_0^D(N)$ is given by
\[
g(X_0^D(N)) = 1 + \frac{\phi(D)\psi(N)}{12} - \frac{e_4(D,N)}{4} - \frac{e_3(D,N)}{3},
\]
where $\phi$ denotes the Euler totient function.
\end{proposition}

\subsection{Atkin--Lehner involutions and Ribet's isogeny}\label{section: AL}

For the entirety of this work, for a curve $X$ over a field $F$, when we write $\Aut(X)$ or refer to the automorphisms of $X$ we mean, unless otherwise specified, the group of automorphisms of $X$ defined over the base field $F$. 

The automorphism group $\Aut(X_0^D(N))$ of the curve $X_0^D(N)$ over $\Q$ has a subgroup consisting of involutions which are natural from the moduli-theoretic perspective. This subgroup is that of the \emph{Atkin--Lehner involutions}, and it contains a non-trivial involution $w_m(D,N)$ for each divisor $m \parallel DN$ with $m\ne 1$, where $m \parallel DN$ means that $m$ is a Hall divisor, i.e., that $m \mid DN$ and $\gcd(m,DN/m) = 1$. We denote this subgroup by $W_0(D,N) \leq \textnormal{Aut}(X_0^D(N))$. For more background on Atkin--Lehner involutions, we refer to \cite{Ogg77} and \cite[\S 2]{Ogg85}; here we give a quick definition from the quaternion arithmetic vantage and state what we need regarding the structure of this group.

Realizing $X_0^D(N)$ as the coarse moduli scheme attached to the moduli problem of parameterizing abelian surfaces with a quaternionic multiplication structure by an Eichler order $\mathcal{O}_N$ of level $N$ in the quaternion algebra $B$ of discriminant $D$, we have the following subgroup of the normalizer of $\mathcal{O}_N$ 
\[N_{{B}^\times_{> 0}}(\mathcal{O}_N) := \{u \in {B}^\times : \text{nrd}(u) > 0 \text{ and } u^{-1} \mathcal{O}_N u = \mathcal{O}_N \},\] 
where $\text{nrd}$ is the reduced norm, and we realize the group of Atkin--Lehner involutions as 
\[ W_0(D,N) := N_{{B}^\times_{> 0}}(\mathcal{O}_N) / \mathbb{Q}^\times \mathcal{O}_N^\times \leq \text{Aut}(X_0^D(N)). \]
Recalling that $\omega(n)$ denotes the number of distinct prime divisors of an integer $n$, this group is a finite abelian $2$-group of the form
\[
W_0(D,N) = \{w_m(D,N) : m \parallel DN\} \cong \left(\Z/2\Z\right)^{\omega(DN)},
\]
though we simply write $w_m$ for $w_m(D,N)$ for brevity when the curve $X_0^D(N)$ whose automorphisms we are considering is clear from context.

For a subgroup $W \leq W_0(D,N)$, we call the quotient $X_0^D(N)/W$ an \emph{Atkin--Lehner quotient} of $X_0^D(N)$. We also refer to the elements of the subgroup of $\Aut\left(X_0^D(N)/W\right)$ induced by the elements of $W_0(D,N)$ as Atkin--Lehner involutions.  

Ogg determined the set of fixed points for each Atkin--Lehner involution on the curve $X_0^D(N)$: these are pairwise disjoint for distinct involutions and they consist of complex multiplication (CM) points by certain imaginary quadratic orders.
\begin{proposition}\cite[p. 283]{Ogg83}\label{prop: Ogg_fixed_pts}
    Let $m>1$ with $m \parallel DN$. The fixed points of the Atkin--Lehner involution $w_m$ acting on $X_0^D(N)$ are points with \textnormal{CM} by the following imaginary quadratic orders:
\[   
R = 
     \begin{cases}
       \Z[i] \text{ and } \Z[\sqrt{-2}] &\quad\textnormal{if } m = 2, \\
       \Z\left[\frac{1+\sqrt{-m}}{2}\right]\text{and } \Z[\sqrt{-m}] &\quad\textnormal{if } m \equiv 3 \;(\bmod \; 4), \\
       \Z[\sqrt{-m}] &\quad\textnormal{ otherwise}.\\
     \end{cases}
\]
\end{proposition}

In the referenced work, Ogg provides the count for the number of fixed points of $w_m$ with CM by $R$ for each of the orders $R$ listed for $m$ in \cref{prop: Ogg_fixed_pts}. Letting $\mathcal{O}_N$ denote a fixed Eichler order of level $N$ in the indefinite quaternion algebra over $\Q$ of discriminant $D$, the count is given as
\[
h(R) \prod_{\substack{p|\frac{DN}{m} \\ \text{prime}}} \nu_p(R,\mathcal{O}_N),
\]
where $h(R)$ is the class number of the order $R$ and $\nu_p(R,\mathcal{O}_N)$ is the number of inequivalent optimal embeddings (see \cite[Definition 30.3.2]{Voight21}) of the localization $R_p$ of $R$ at $p$ into the localization $(\mathcal{O}_N)_p$ of $\mathcal{O}_N$ at $p$. The counts of these optimal embeddings are given in \cite[Theorem 2]{Ogg83}.

Using \cref{prop: Ogg_fixed_pts} and the Riemann--Hurwitz theorem in combination with the genus formula for $X_0^D(N)$ (\cref{prop: genus_formula}), one can compute the genus of every Atkin--Lehner quotient of $X_0^D(N)$. A function handling these genus computations is implemented in the file \texttt{quot{\_}genus.m} in \cite{MPSSRep}. 

The following theorem of Ribet, relating the Jacobian of $X_0^D(N)$ to a sub-quotient of the Jacobian of the modular curve $X_0(DN)$, is crucial to our implementation of a finite field point count algorithm for $X_0^D(N)$ in \cref{section: point_counts}. While Ribet's original result in \cite{Ribet80}, as well as the work of Bertolini--Darmon \cite[\S 1]{BD96} in which our statements regarding the Atkin--Lehner action appear, each contain a squarefree restriction on the level $N$, this hypothesis can be removed. One can find the relevant results on bases of spaces of modular forms and embeddings of quaternion orders in the work of Hijikata--Pizer--Shemanske \cite{HPS89a, HPS89b}, and in particular we refer the reader to the relevant statement reproved by Martin in \cite[Theorem 1.1 (ii)]{Martin}. One can also see \cite[Corollary 3.3]{Roberts} for an alternate proof of the general version.

\begin{theorem}\cite{Ribet80, BD96, HPS89a, HPS89b}\label{Ribet_isog}
Let $w_m(1,DN)$ denote the involution on $X_0(DN)$ and let $w_m(D,N)$ denote the involution on $X_0^D(N)$ corresponding to a Hall divisor $m$ of $DN$. There is an isogeny defined over~$\Q$
\[
\Psi \colon \Jac(X_0(DN))^{D-\new} \longrightarrow \Jac(X_0^D(N))
\]
such that the following holds: for each $m \parallel DN$ and each isogeny factor $A_f$ of $\Jac(X_0^D(N))$, letting $\epsilon_f(D,N)$ and $\epsilon_f(1,DN)$ denote, respectively, the signs of the actions of $w_m(D,N)$ and $w_m(1,DN)$ on $A_f$, we have 
\begin{equation}\label{eqn: AL_action_Ribet}
   \epsilon_f(D,N) = (-1)^{\omega(\gcd(D,m))} \epsilon_f(1,DN).
\end{equation} 
\end{theorem}

\begin{remark}
The Jacobian $\Jac(X_0(DN))$ of $X_0(DN)$ decomposes into irreducible isogeny factors $A_f$, which correspond to distinguished cusp forms $f$. To obtain the $D$-new part, we take those factors with conductor divisible by $D$. Since $A_f$, when $f$ is a newform of level $DN$, has conductor $(DN)^{\dim A_f}$ (see \cite{KW09a} and \cite{KW09b} for the proof of Serre's modularity conjecture, and \cite[Théorème~5]{Ser87} for its consequence we are interested in), 
hence, by the classical theory of modular curves and oldforms (see \cite[Theorem~5]{AL70} and \cite[Chapters~5 and 6]{DS05}), we have
\begin{equation*}
\Jac(X_0(DN))^{D-\new} \sim\prod_{n|N}\Jac^{\new}(X_0(Dn))^{\sigma_0\left(\frac{N}{n}\right)},
\end{equation*}
where $\sigma_0(m)$ is the sum of positive divisors of the integer $m$ and $\Jac^{\new}(X_0(Dn))$ denotes the new part of $\Jac(X_0(Dn))$, i.e., the part where the isogeny factors $A_f$ are associated to newforms $f$ of level $Dn$. 

The isogeny factors of the Jacobian of an Atkin--Lehner quotient $X_0^D(N)/W$, and the corresponding cusp forms, can then be determined by the action of $W$ on these isogeny factors as given in \cref{eqn: AL_action_Ribet}. 
\end{remark}

\subsection{Known results on automorphism groups}\label{subsection: auts}

It is expected that
\begin{equation}\label{eq:aut_conj}
\textnormal{Aut}(X_0^D(N)) = W_0(D,N)
\end{equation}
for all but finitely many Shimura curves $X_0^D(N)$ with squarefree level $N$ (see \cite[Conjecture 1.1]{KR08}). This is known if one restricts to the case $D=1$ of the modular curves $X_0(N)$, even for general $N$, by work of Ogg \cite{Ogg77}, Kenku--Momose \cite{KM88}, Elkies \cite{Elkies90}, and Harrison \cite{Har11}. Here, the only curves $X_0(N)$ of genus at least~$2$ with automorphisms which are not Atkin--Lehner involutions are those of level $N \in \{ 37, 63,108 \}$. 

In \cref{section: automorphisms}, we provide methods for checking this equality, along with related experimental results for $D>1$. Here, we recall related results of Kontogeorgis--Rotger. 

\begin{proposition}\cite[Proposition 1.5]{KR08}\label{prop: KR_quotient}
Suppose $N$ is squarefree and coprime to $D$ and let $W \leq W_0(D,N)$. If $g(X_0^D(N)/W) \geq 2$, then all automorphisms of $X_0^D(N)/W$ over $\overline{\Q}$ are defined over $\Q$ and 
\[ \textnormal{Aut}(X_0^D(N)/W) \cong \left(\Z/2\Z\right)^s \] 
for some $s \geq \omega(DN) - \textnormal{ord}_2\left(|W|\right)$. 
\end{proposition}

\begin{proposition}\cite[Corollary 1.8]{KR08}\label{prop: KR_top_curve}
Suppose $N$ is squarefree and coprime to $D$. If $g(X_0^D(N)) \geq 2$, then 
\[ \textnormal{Aut}(X_0^D(N)) \cong \left(\Z/2\Z\right)^s \] 
for some $s \in \{\omega(DN), \omega(DN)+1\}$. 
\end{proposition}

In the following lemma, giving sufficient criteria for the equality $\Aut(X_0^D(N)) = W_0(D,N)$, the numbers $e_k(D,N)$ for $k \in \{3,4\}$ are as in \cref{prop: genus_formula}.

\begin{lemma}\cite[Theorems 1.6 and 1.7]{KR08}\label{lemma: KR_all_atkin_lehner}
Suppose that $N$ is squarefree and coprime to $D$ and that $g \colonequals g(X_0^D(N)) \geq 2$. If at least one of the following statements holds:
\begin{enumerate}
    \item $e_3(D,N) = e_4(D,N) = 0$;
    \item $3 \mid DN$, for all primes $p \mid N$ we have $\left(\frac{-3}{p}\right) \ne -1$, and for at most one prime $\ell \mid D$ we have $\left(\frac{-3}{\ell}\right) = 1$;
    \item $\omega(DN) = \textnormal{ord}_2(g-1)+2$;
    \item there is a prime $p \nmid 2DN$ such that $\omega(DN) = \textnormal{ord}_2\left(|X_0^D(N)(\F_p)|\right) + 1$;
\end{enumerate}
then $\Aut(X_0^D(N)) = W_0(D,N)$. 
\end{lemma}


\section{Gonalities}\label{section: gonalities}

\subsection{Basic properties}

\begin{definition}
Let $F$ be a perfect field with algebraic closure $\overline{F}$, and let $X$ be a curve over $F$. The \emph{gonality} $\gamma_F(X)$ of $X$ (over $F$) is the least degree of a map from $X$ to the projective line $\P^1_F$ over $F$. A map $X \to \mathbb{P}^1_F$ over $F$ of degree $\gamma_F(X)$ is called a \emph{gonal map} for $X$. 
\end{definition}

For $L/F$ a field extension, we have the base change of $X$ to $L$
\[ X_L := X \otimes_{\Spec K} \Spec L, \]
and we denote the gonality of this base change by
\[ \gamma_L(X) := \gamma_L(X_L). \]
In particular, the \emph{geometric gonality} of $X$ is defined to be $\gamma_{\overline{F}}(X)$. We recall the following general facts about gonalities which we use often.

\begin{proposition}\cite[Proposition A.1]{Poonen}\label{Prop: Poonen_gonality}
Let $X$ be a curve of genus $g$ over a field $F$. 
\begin{enumerate}
    \item If $L/F$ is a field extension, then $\gamma_L(X) \leq \gamma_F(X)$. 
    \item If $F$ is algebraically closed and $L/F$ is a field extension, then $\gamma_L(X) = \gamma_F(X)$. 
    \item If $g>1,$ then $\gamma_F(X) \leq 2g-2$. For each $g>1$, there exist $F$ and $X/F$ for which equality holds. 
    \item If $X(F) \neq \varnothing$, then $\gamma_F(X) \leq g+1$. If further $g \geq 2$, then $\gamma_F(X) \leq g$ (and, again, equality is possible for each $g$). 
    \item If $F$ is algebraically closed, then $\gamma_F(X) \leq \lfloor \frac{g+3}{2} \rfloor$. Equality holds for a general curve of genus~$g$ over~$F$.
    \item If $\pi\colon X \to Y$ is a dominant rational map of curves over~$F$, then 
    \[
    \gamma_F(Y) \leq \gamma_F(X)\leq \textnormal{deg}(\pi)\cdot \gamma_F(Y).
    \]
\end{enumerate}
\end{proposition}

\subsection{Gonalities and reduction}
Let $X$ be a curve over $\Q$, let $p$ be a prime of good reduction for $X$, and let $q=p^r$ where $r$ is a positive integer. We let $\gamma_{\F_q}(X)$ denote the gonality of the base change of $X$ to $\F_q$. In this setting, we have the inequality
\[ \gamma_{\F_q}(X) \leq \gamma_{\Q}(X), \]
as from a rational function on $X$ over $\Q$ corresponding to a gonal map one obtains via base change a rational function over $\F_q$ of degree at most $\gamma_\Q(X)$. 
In particular, the Shimura curve $X_0^D(N)$ has good reduction at primes $p \nmid DN$, and, hence, if $q = p^r$, we get 
\[
\gamma_{\F_q}(X_0^D(N)) \leq \gamma_{\Q}(X_0^D(N)).
\]

If $\widetilde{X}$ is a curve over $\F_q$ with gonality $d$ and $f$ is a gonal map for $\widetilde{X}$, then there are at most $d$ points of degree $1$ in the fiber of $f$ above each of the $q+1$ rational points on $\mathbb{P}^1_{\F_q}$. Therefore, we have the following gonality bounds via point counts. 

\begin{lemma}\label{lemma: gonality_bound_by_point_counts}
Let $X$ be a curve over $\Q$, let $p$ be a prime of good reduction for $X$, and let $q = p^r$ for $r$ a positive integer. We have
\[ \gamma_{\Q}(X) \geq \gamma_{\F_q}(X) \geq \dfrac{\#X(\F_q)}{q+1}. \]
\end{lemma}


\section{Finite fields point counts}\label{section: point_counts}

In this section, we describe \cref{algorithm: introduction_version}, which is used for our computations of point counts. It is a variation of the algorithms contained in \cite{DLMS} and \cite{DLMS26}.

By Ribet's isogeny (\cref{Ribet_isog} above), we can write the Jacobian of the Shimura curve $X_0^D(N)$ in terms of isogeny factors of the Jacobian of the modular curve $X_0(DN)$.
In order to work with the latter object, we recall the following results from \cite{DLMS} and \cite{DLMS26}.

\begin{lemma}\label{prop4.4_DLMS}
Let $\mathcal F$ be a finite subset of the set of the Galois orbits of normalized eigenforms of weight 2 of $\bigcup_{n>0}\mathcal S_2^{\textnormal{new}}(\Gamma_0(n))$, and let $J:=\prod_{[f]\in\mathcal F}A_f^{m_f}$, where $A_f$ is the abelian variety associated to Galois orbit $[f]$ of $f$ (see \cite[Definition~6.6.3]{DS05}) and $m_f>0$ being an integer. Moreover, let $V_\ell=\mathrm{Ta}_\ell(J)\otimes_{\Z_\ell} \Q_\ell$ be the $\Q_\ell$-vector space associated to the Tate module of $J$ for a prime $\ell$. 
Then, the characteristic polynomial of $\mathrm{Frob}_p$ acting on $V_\ell$ is 
$$
\prod_{[f]\in\mathcal F}\prod_{h\in [f]} (x^2-a_p(h)x+p)^{m_f},
$$
where $a_p(h)$ is the $p$-th Fourier coefficient of~$h$.
\end{lemma}
\begin{proof}
See \cite[Proposition 4.4]{DLMS}.
\end{proof}

To apply the previous lemma, we need to compute $a_p(h)$ and $m_f$. For the former, we use the data available in \cite{LMFDB}. For the latter, we use the following result.

\begin{lemma}\label{theo2.2_DLMS26}
Given a modular curve $X_0(N)$, let $W$ be its group of Atkin--Lehner involutions. For each subgroup $K \leq W$, the Jacobian of the corresponding quotient satisfies the following isogeny relation over $\Q$
\[
\mathrm{Jac}(X_0(N)/K)\sim \textstyle\prod_f A_f^{m_f},
\]
where $f$ ranges among the newforms of weight $2$ and level $d$ for divisors $d \mid N$ and the multiplicities $m_f$ are as follows: let $K^\perp$ be the set of characters of $W$ that are trivial on $K$ and, for each $f$, denote $v:=\textnormal{ord}_\ell(N/d)$. Then
\[
m_f:=\sum_{\chi\in K^\perp}   \prod_{\ell^e \parallel N} \left(\tfrac{v+1}{2}+\varepsilon_{f,\ell^e}\chi(w_{\ell^e})\tfrac{1+(-1)^{v}}{4} \right),
\]
with the indices $\ell^e$ varying among all prime powers exactly dividing $N$ and $\varepsilon_{f,\ell^e}$ being the eigenvalue of the Atkin--Lehner involution $w_{\ell^e}$ applied to~$f$.
\end{lemma}
\begin{proof}
This follows from \cite[Theorem 2.2]{DLMS26}, taking $n_0=N$ and $n_{\textnormal{ns}}=1$.
\end{proof}

The steps of \cref{algorithm: introduction_version} for counting the number of points of $X_0^D(N)/W$ over the finite field $\F_q$, with $q=p^k$ and $p\nmid DN$, are as follows.
\begin{enumerate}
    \item For every positive integer $n\mid N$, choose a basis $\mathcal B(nD)$ of newforms of level $nD$ and let $\mathcal B:=\bigcup_{n\mid N} \mathcal B(nD)$ be the union of these bases.
    \item For every $f \in\mathcal B$, compute the eigenvalues with respect to the Atkin--Lehner operators $w_{\ell^e}$, for $\ell^e||DN$ (these data are available on the database \cite{LMFDB} for $DN\leq 10000$) and, then, use them for computing the multiplicity $m_f$ as in \cref{theo2.2_DLMS26}. We point out that we need to switch $\varepsilon_{f,\ell^e}$ with $-\varepsilon_{f,\ell^e}$ if $\ell\mid D$ by \cref{Ribet_isog}.
    \item Compute the Hecke eigenvalue $a_p(f)$ for every $f\in\mathcal B$ (these data are also available on the database \cite{LMFDB} for $DN\leq 10000$).
    \item If $k=1$, compute the Frobenius traces acting on $A_f$: 
    \[
    \mathrm{tr}(\mathrm{Frob}_p|A_f)=\sum_{h\in [f]} a_p(h),
    \]
    where $[f]$ is the Galois orbit of $f$.
    \item If $k>1$, for each $h\in [f]$, compute the (complex) roots $\alpha(h)$ and $\beta(h)$ of the polynomial $x^2-a_p(h)x+p=0$. Then compute
    \[
    \mathrm{tr}(\mathrm{Frob}_q|A_f)=\sum_{h\in [f]}\alpha(h)^k+\beta(h)^k.
    \]
    \item Finally, compute
    \[
    \#X_0^D(N)(\F_q)=q+1-\sum_{[f]} m_f \mathrm{tr}(\mathrm{Frob}_q|A_f),
    \]
    where the sum is taken over the distinct Galois orbits for $f\in \mathcal{B}$ and $m_f$ is defined in \Cref{theo2.2_DLMS26}.
\end{enumerate}


\section{Automorphisms of Shimura curves}\label{section: automorphisms}

Let $X/\Q$ be a curve and let $p$ be a prime of good reduction for~$X$. The automorphism group of $X$ over~$\Q$ injects into the automorphism group of the reduction of $X$ modulo~$p$ \cite[Proposition~3.38, Chapter~10]{Liu}. In particular, if $p \nmid DN$, then the automorphism group of the curve $X_0^D(N)$ over~$\Q$ injects into that of its reduction modulo~$p$.

Point counts over the extensions $\F_q$ for $q$ a power of $p$ yield constraints on the automorphism group of a curve over~$\F_p$. For instance, a constraint of this type, specific to the curves $X_0^D(N)$, is seen in Part~(4) of \cref{lemma: KR_all_atkin_lehner}. In this section, we use our point counts algorithm to implement checks on the automorphism group of $X_0^D(N)$ and compute evidence towards the conjecture mentioned in \cref{subsection: auts} (see \Cref{eq:aut_conj}).

By \cref{prop: KR_quotient}, we know that when $N$ is squarefree and coprime to~$D$ the automorphism group of every Atkin--Lehner quotient of $X_0^D(N)$ 
consists of involutions. Of course, every Atkin--Lehner quotient of $X_0^D(N)$ has a non-trivial Atkin--Lehner involution \emph{except possibly} for the full quotient $X_0^D(N)^* := X_0^D(N)/W_0(D,N)$. We expect that this star quotient has trivial automorphism group for all but finitely many pairs $(D,N)$, and the following instance of a result of Gonz\'{a}lez is a main tool we use to show this for a fixed pair.

\begin{lemma}\cite[Equation (2.2)]{Gonzalez17}\label{lemma: Gonzalez_no_involutions}
Let $q = p^r$ be a prime power, let $X$ be a curve of genus $g \geq 2$ over the finite field $\F_q$, and suppose that $X$ has a non-trivial automorphism of prime order $\ell$ over $\F_q$. For each integer $n \geq 1$, let 
\[ A_n := X(\F_{q^n}) \setminus \displaystyle\bigcup_{i=1}^{n-1} X(\F_{q^i}). \] 
Then 
\[ \sum_{n \geq 1, \gcd(n,\ell) = 1} n\cdot \textnormal{mod}\left(|A_n|/n,\ell\right) \leq \left\lfloor \dfrac{2g}{\ell-1} \right\rfloor + 2, \]
where $\textnormal{mod}\left(m,\ell\right)$ denotes the remainder upon division of $m$ by $\ell$. 
\end{lemma}

To state our main computational results on automorphism groups, we set the following notation: let $\mathcal{S}$ denote the set of all pairs $(D,N)$ consisting of an indefinite quaternion discriminant $D>1$ over $\Q$ and a squarefree positive integer $N$ coprime to $D$, such that 
\begin{enumerate}
    \item $g(X_0^D(N)^*) > 2$ and
    \item $DN \leq 10000$. 
\end{enumerate}
The set $\mathcal{S}$ has cardinality $|\mathcal{S}| = 10609$. 

\begin{remark}
The hypothesis (1) on the genus is necessary, as any curve $X_0^D(N)^*$ of genus $g \leq 2$ has a non-trivial automorphism, and this only excludes finitely many pairs $(D,N)$ (see \cref{section: tetragonal} for the relevant argument). Of course, we can only compute point counts for a finite list of curves, and the specific restriction (2) of $DN \leq 10000$ reflects the limits of the data we use of the eigenvalues of the Atkin--Lehner operators acting on newforms from the L-Functions and Modular Forms Database \cite{LMFDB}, as discussed in \cref{section: point_counts}.
\end{remark}

\begin{theorem}\label{theorem: star_auts}
The curve $X_0^D(N)^*$ has trivial automorphism group for all pairs $(D,N) \in \mathcal{S}$ except possibly for the $1321$ pairs listed in \cref{table: unknown_automorphism_group_stars}. 
\end{theorem}
\begin{proof}
By \cref{prop: KR_quotient}, the only possible non-trivial automorphisms of $X_0^D(N)^*$ for $(D,N) \in \mathcal{S}$ are involutions. \cref{lemma: Gonzalez_no_involutions} shows that $\Aut(X_0^D(N)^*)$ is trivial for $9288$ elements of $\mathcal{S}$. This is computed using our point counts algorithm, applying the lemma of Gonz\'{a}lez using prime powers $q = p^r$ with $p<100$ a prime of good reduction for $X_0^D(N)$ and with powers $r$ up to $100$ when possible, and more generally as high as the pre-computed trace data we use allows. The code for these computations can be found in the file \texttt{aut{\_}checks.m} in \cite{MPSSRep}.
\end{proof}

\begin{corollary}\label{corollary: all_atkin_lehner}
\phantom{a}
\begin{enumerate}
    \item Let $(D,N) \in \mathcal{S}$ be a pair which does not appear among the $1321$ pairs in \cref{table: unknown_automorphism_group_stars} and let $W \leq W_0(D,N)$. Then $\Aut(X_0^D(N)/W)$ consists only of Atkin--Lehner involutions, and hence
\[ \Aut(X_0^D(N)/W) \cong \left(\Z/2\Z\right)^{\omega(DN) - \textnormal{ord}_2\left(|W|\right)}. \]
\item Let $(D,N) \in \mathcal{S}$ be a pair which does not appear among the $12$ pairs in \cref{table: unknown_automorphism_group}. Then $\Aut(X_0^D(N))$ consists only of Atkin--Lehner involutions, i.e., 
\[ \Aut(X_0^D(N)) = W_0(D,N) \cong \left(\Z/2\Z\right)^{\omega(DN)}. \]
\end{enumerate}
\end{corollary}
\begin{proof}
By \cref{prop: KR_quotient} we know that $\Aut(X_0^D(N)/W)$ is commutative. Therefore, if this group contains a non-Atkin--Lehner involution, then this involution would induce a non-trivial involution on the quotient $X_0^D(N)^*$. If $(D,N)$ is a pair in \cref{table: unknown_automorphism_group_stars}, then this would contradict \cref{theorem: star_auts}, giving Part~(1). 
Part~(2) then follows quickly from Part~(1) by supplementing it with applications of \cref{lemma: KR_all_atkin_lehner} to the curves $X_0^D(N)$ for pairs $(D,N)$ in \cref{table: unknown_automorphism_group_stars}. Code for these computations is again found in \texttt{aut{\_}checks.m} in \cite{MPSSRep}. 
\end{proof}

\begin{remark}
The computations performed for the proofs of \cref{theorem: star_auts} and \cref{corollary: all_atkin_lehner} took roughly $1.66$ hours of computational time on a 2020 m1 MacBook Pro. Access to further trace data for computing finite field point counts would allow for a similar result for curves up to a larger level bound, and would possibly allow us to improve our result in the level range $DN \leq 10000$ as well by allowing for computations of point counts over finite fields of larger cardinality. 

We also note that the squarefree restriction on $N$ put on pairs in $\mathcal{S}$ is \emph{not} necessary in the use of our point count algorithm and \cref{lemma: Gonzalez_no_involutions}. Rather, it is only needed in the use of \cref{lemma: KR_all_atkin_lehner} and in the transition from \cref{theorem: star_auts} to \cref{corollary: all_atkin_lehner} using \cref{prop: KR_quotient}. We use these individual results more generally in \cref{section: tetragonal}, and functions applying each are available in the file \texttt{aut{\_}checks.m} in \cite{MPSSRep}. We restrict to the set $\mathcal{S}$ for the above results as a sample of the power of our point counts in determining automorphism groups. 
\end{remark}


\section{Tetragonal Shimura curves}\label{section: tetragonal}

\cref{lemma: gonality_bound_by_point_counts} provides a way to obtain lower bounds on the gonality $\gamma_\Q(X_0^D(N))$ of $X_0^D(N)$ via point counts over finite fields. In this section, we apply this method in combination with other techniques to work towards determining all of the tetragonal and geometrically tetragonal Shimura curves $X_0^D(N)$. We begin by defining these notions and similarly making clear our convention for the terms hyperelliptic, bielliptic, tetragonal, and sub-hyperelliptic. 

\begin{definition}
Let $F$ be a perfect field and let $X$ be a curve over $F$ of genus $g(X) \geq 2$. We call $X$
\begin{itemize}
    \item \emph{hyperelliptic} if there exists a degree $2$ map from $X$ to $\mathbb{P}^1_F$ over $F$,
    \item \emph{geometrically hyperelliptic} if its base change to $\overline{F}$ is hyperelliptic,
    \item \emph{bielliptic} if there exists a degree $2$ map from $X$ to an elliptic curve over $F$,
    \item \emph{geometrically bielliptic} if its base change to $\overline{F}$ is bielliptic,
    \item \emph{tetragonal} if there exists a degree $4$ map from $X$ to $\mathbb{P}^1_F$ over $F$, and 
    \item \emph{geometrically tetragonal} if its base change to $\overline{F}$ is tetragonal. 
\end{itemize}
\end{definition}

\begin{remark}
Note that our definition of tetragonal simply means that the curve admits a degree~$4$ map to the projective line; it does not mean that the curve has gonality \emph{exactly}~$4$, though it implies that it has gonality $2, 3,$ or $4$.
\end{remark}

\begin{definition}
Let $F$ be a perfect field. We call a curve $X$ over $F$ \emph{sub-hyperelliptic} if it is hyperelliptic or has genus $0$ or $1$. We call $X$ \emph{geometrically sub-hyperelliptic} if its base change to $\overline{F}$ is sub-hyperelliptic. 
\end{definition}

We recall in the next two subsections the Castelnuovo--Severi criterion and some quick consequences as well as prior work on curves $X_0^D(N)$ of low gonality. These provide methods to narrow lists of candidate curves $X_0^D(N)$ prior to our application of point counts to a select finite list of candidates in the final subsection. 

\subsection{The Castelnuovo--Severi inequality}\label{subsection: CS}

The Castelnuovo--Severi inequality is one of our main tools in proving certain curves $X_0^D(N)$ are not tetragonal. A proof of the version below, working over a perfect field, can be found in \cite[Theorem 14]{KS24}.

\begin{theorem}[Castelnuovo--Severi]\label{theorem: CS}
Let $F$ be a perfect field and let $X$, $Y$, and $Z$ be curves over $F$. Let
\[
\pi_Y\colon X \to Y\qquad\text{and} \qquad \pi_Z\colon X \to Z
\]
be non-constant morphisms defined over $F$. Then either
\[
g(X) \leq \textnormal{deg}(\pi_Y)\cdot g(Y) + \textnormal{deg}(\pi_Z)\cdot g(Z) + (\textnormal{deg}(\pi_Y)-1)(\textnormal{deg}(\pi_Z)-1),
\]
or there exists a curve $X'$ over $F$ and a morphism $X \to X'$ over $F$ of degree greater than $1$ through which both $\pi_Y$ and $\pi_Z$ factor. 
\end{theorem}

We next collect applications of \cref{theorem: CS} that we will use in what follows. 

\begin{lemma}\label{lemma: CS_unique_hyperelliptic}
Let $X$ be a geometrically hyperelliptic curve over a perfect field $F$ and let $f_1 \colon X \to C_1$ and $f_2 \colon X \to C_2$ be degree~$2$ maps over~$F$ to conics $C_1$ and $C_2$. Then there exists an isomorphism $f\colon C_1 \to C_2$ over~$F$ with $f_2 = f \circ f_1$.
\end{lemma}
\begin{proof}
Note that the hypothesis that $X$ is geometrically hyperelliptic implies that $X$ has genus $g \geq 2$. \cref{theorem: CS} then implies that $f_1$ and $f_2$ factor through a common degree~$2$ map $h \colon X \to X'$ over~$F$. We then have isomorphisms $f_1' \colon X' \to C_1$ and $f_2' \colon X' \to C_2$ over~$F$ so that $f_i = f_i' \circ h$. The claim then follows by letting $f = f_2' \circ (f_1')^{-1}$. 
\end{proof}

\begin{lemma}\label{lemma: CS_geom_hyperelliptic_bound}
Let $X$ be a curve of genus~$g$ over a perfect field~$F$ possessing an involution~$\sigma$, 
and suppose that the genus $g_\sigma$ of the quotient $X/\langle \sigma \rangle$ satisfies $g_\sigma > 0$. If
\[
g > 2g_\sigma + 1,
\]
then $X$ is not geometrically hyperelliptic. 
\end{lemma}
\begin{proof}
Let $q\colon X \to X/\langle \sigma \rangle$ denote the natural quotient map. Suppose to the contrary that there is a degree~$2$ map $f\colon X \to C$ to a conic~$C$. Because of the assumed strict inequality, \cref{theorem: CS} tells us that we must have a commutative diagram of the following shape over~$F$
\[
\begin{tikzpicture}[node distance=2.2cm]
\node (X) at (0,2) {$X$};
\node (Xp) at (0,0) {$X'$};
\node (Q) at (-3,-2) {$X/\langle \sigma \rangle$};
\node (C) at (3,-2) {$C$};

\draw[->] (X) -- node[right] {$\pi$} (Xp);
\draw[->] (X) -- node[left] {$q$} (Q);
\draw[->] (X) -- node[right] {$\;f,\, \deg(f) = 2$} (C);
\draw[->] (Xp) -- node[below] {$q'$} (Q);
\draw[->] (Xp) -- node[below] {$f'\;\;$} (C);
\end{tikzpicture}
\]
with $\deg(\pi) > 1$. Hence $\deg(\pi) = 2$ and $q'$ and $f'$ are both isomorphisms, which implies that 
\[ g\left(X/\langle \sigma \rangle\right) = g(C) = 0 \]
and contradicts our hypothesis. 
\end{proof}

\begin{lemma}\label{lemma: CS_unique_bielliptic}
Let $X$ be a geometrically bielliptic curve of genus $g \geq 6$ over a perfect field~$F$ and let $f_1 \colon X \to C_1$ and $f_2 \colon X \to C_2$ be degree~$2$ maps over~$F$ to genus~$1$ curves $C_1$ and $C_2$. Then there exists an isomorphism $f\colon C_1 \to C_2$ over~$F$ with $f_2 = f \circ f_1$.
\end{lemma}
\begin{proof}
By the genus assumption on~$X$, \cref{theorem: CS} then implies that $f_1$ and $f_2$ factor through a common degree~$2$ map $h \colon X \to X'$ over~$F$. The proof then proceeds just as in \cref{lemma: CS_unique_hyperelliptic}.
\end{proof}

\begin{lemma}\label{lemma: CS_tetragonal_with_involution}
Let $X$ be a geometrically tetragonal curve of genus~$g$ over a perfect field~$F$ with a geometrically tetragonal map $f\colon X \to C$ over $F$ to a conic~$C$. Let $\sigma$ be an involution on~$X$,
and let $g_\sigma$ denote the genus of $X/\langle \sigma \rangle$. If 
\begin{equation}\label{CS_tetragonal} 
g > 2g_\sigma + 3,
\end{equation}
then $X/\langle \sigma \rangle$ is geometrically sub-hyperelliptic. 
\end{lemma}
\begin{proof}
Let $q\colon X \to X/\langle \sigma \rangle$ denote the natural quotient map. By \cref{theorem: CS}, the assumed inequality implies that we have a commutative diagram over~$F$ of the following form:
\[
\begin{tikzpicture}[node distance=2.2cm]
\node (X) at (0,2) {$X$};
\node (Xp) at (0,0) {$X'$};
\node (Q) at (-3,-2) {$X/\langle \sigma \rangle$};
\node (C) at (3,-2) {$C$};

\draw[->] (X) -- node[right] {$\pi$} (Xp);
\draw[->] (X) -- node[left] {$q$} (Q);
\draw[->] (X) -- node[right] {$\;f,\, \deg(f) = 4$} (C);
\draw[->] (Xp) -- node[below] {$q'$} (Q);
\draw[->] (Xp) -- node[below] {$f'\;\;$} (C);
\end{tikzpicture}
\]
with $\deg(\pi) = 2$. Here, $q'$ must be an isomorphism, and thus
\[ f' \circ (q')^{-1} \colon X/\langle \sigma \rangle \longrightarrow C \]
is a degree $2$ map, proving the claim.
\end{proof}

\begin{lemma}\label{lemma: CS_unique_tetragonal}
Let $X$ be a geometrically tetragonal curve of genus~$g$ over a perfect field~$F$, with geometrically tetragonal maps $f_1 \colon X \to C_1$ and $f_2\colon X \to C_2$ to conics $C_1$ and $C_2$ over~$F$. If $g \geq 10$, then either
\begin{itemize}
    \item $X$ is geometrically hyperelliptic,
    \item $X$ is geometrically bielliptic, or
    \item there exists an isomorphism $f\colon C_1 \to C_2$ over $F$ with $f_2 = f \circ f_1$.
\end{itemize}
\end{lemma}
\begin{proof}
By \cref{theorem: CS}, the maps $f_1$ and $f_2$ are either equal up to composition with an isomorphism between $C_1$ and $C_2$ over~$F$ or they factor over $F$ through a common degree~$2$ map $h \colon X \to X'$. In the latter case, the curve $X'$ is evidently geometrically sub-hyperelliptic over~$F$. In particular, if $X'$ has genus~$0$, then $X$ is geometrically hyperelliptic, and if $X'$ has genus~$1$, then $X$ is geometrically bielliptic. 

The only remaining possibility is that $X'$ is geometrically hyperelliptic with genus $g' \geq 2$. In this case, applying \cref{lemma: CS_unique_hyperelliptic}, we get that $f_1$ and $f_2$ are equal up to composition with an isomorphism $C_1 \to C_2$ over~$F$.
\end{proof}

\begin{corollary}\label{corollary: CS_not_tetragonal}
Let $X$ be a geometrically tetragonal curve of genus $g \geq 10$ over a perfect field~$F$. Assume that $X$ is neither geometrically hyperelliptic nor geometrically bielliptic and that there exists a degree~$4$ map $f \colon X \to C$ over~$F$ to a conic~$C$ with $C(F) = \varnothing$. Then $X$ is not tetragonal over~$F$. 
\end{corollary}
\begin{proof}
By \cref{lemma: CS_unique_tetragonal}, we know with our hypotheses that if $h \colon X \to C'$ is another degree~$4$ map over~$F$ to a conic~$C'$, then $C$ and $C'$ are isomorphic over~$F$. Therefore, $C'$ is not isomorphic to $\P^1_F$. 
\end{proof}

The next result is useful in the application of  \cref{corollary: CS_not_tetragonal}.

\begin{lemma}\label{lemma: hyperelliptic_bielliptic_point_count_bounds}
Let $X$ be a curve over~$\Q$ and let $q = p^r$ a power of a prime~$p$ of good reduction for~$X$. 
\begin{enumerate}
    \item If $X$ is geometrically hyperelliptic, then
\[ X(\F_q) \leq 2q+2. \]
    \item If $X$ is geometrically bielliptic of genus $g \geq 6$, then 
    \[ X(\F_q) \leq 2q+2+4\sqrt{q}. \]
\end{enumerate}
\end{lemma}
\begin{proof}\phantom{a}
\begin{enumerate}
\item If the statement said ``hyperelliptic over~$\Q$'' rather than ``geometrically hyperelliptic,'' then it would follow immediately from \cref{lemma: gonality_bound_by_point_counts}. The point here is that the uniqueness of a geometrically hyperelliptic involution (as follows from \cref{lemma: CS_unique_hyperelliptic}) makes the result true as stated. The geometrically hyperelliptic involution on $X$ must be defined over~$\Q$ by uniqueness and, hence, $X$ has a degree~$2$ map to a conic~$C$ over~$\Q$ (which may have no rational points). Therefore, we get a map from $X_{\F_q}$ to a conic $\widetilde{C}$ over~$\F_q$, as the automorphism group of $X$ over~$\Q$ injects into that of the reduction modulo~$p$ \cite[Proposition~3.38, Chapter~10]{Liu}. We know that $\widetilde{C} \cong \mathbb{P}^1_{\F_q}$, as $\widetilde{C}$ is a smooth conic over a finite field, so $X_{\F_q}$ is hyperelliptic over~$\F_q$ and the argument from \cref{lemma: gonality_bound_by_point_counts} applies. 
\item This part is similar. Here, our assumption on the genus provides that the geometrically bielliptic involution on $X$ is unique by \cref{lemma: CS_unique_bielliptic}. We then have a degree~$2$ map $X \to C$ over~$\Q$ where $C$ is a smooth genus~$1$ curve. This in turn gives a degree~$2$ map over~$\F_q$ from $X_{\F_q}$ to a smooth genus~$1$ curve $\widetilde{C}$. By the Hasse bound, $\widetilde{C}$ has at most $q+1+2\sqrt{q}$ over~$\F_q$, and there are at most~$2$ $\F_q$-points in the fiber above each under the degree~$2$ map from~$X_{\F_q}$. 
\end{enumerate}
\end{proof}

The following consequence of the Riemann--Hurwitz theorem is helpful in the application of \cref{lemma: hyperelliptic_bielliptic_point_count_bounds}.

\begin{proposition}\cite[Proposition~4]{Hasegawa97}\label{proposition: not_hyperelliptic}
Let $X/\C$ be a hyperelliptic curve of genus~$g$. Let $\sigma$ be an involution on $X$ 
and let $g_\sigma$ be the genus of $X/\langle \sigma \rangle$. Suppose $g_\sigma > 0$. If $g$ is even, then $g_\sigma = g/2$. If $g$ is odd, then $g_\sigma \in \left\{  \frac{g-1}{2} , \frac{g + 1}{2}\right\}$. 
\end{proposition}

Finally, the following result of Najman--Orli\'{c}, which is a corollary of the Tower Theorem (see \cite[Theorem 4.4]{NO24}), will be useful in proving that a few curves are not geometrically tetragonal. 
\begin{proposition}\cite[Corollary 4.6]{NO24}\label{prop: NO_corollary}
\begin{enumerate}
    \item Let $C$ be a curve over $\Q$ of genus $g \geq 5$ which is geometrically trigonal and such that $C(\Q) \neq \emptyset$. Then $C$ is trigonal over $\Q$.
    \item Let $C$ be a curve over $\Q$ of genus $g \geq 10$ with $\gamma_{\Qbar}(X) = 4$ and such that $C(\Q) \neq \emptyset$. Then $\gamma_\Q(X) = 4$. 
\end{enumerate}
\end{proposition}

\subsection{Known results on Shimura curves of low gonality}

The complete list of curves $X_0^D(N)$ with $D>1$ and $\gcd(D,N) = 1$ of gonality $\gamma_\Q(X_0^D(N)) \leq 2$ and genus at least~$2$ is known. Equivalently, all such curves satisfying $\gamma_\Q(X_0^D(N)) \leq 3$ are known, since $\Q$-gonality of these curves is always even as $X_0^D(N)(\R) = \varnothing$ for $D>1$. So too are those of geometric gonality $\gamma_{\Qbar}(X_0^D(N)) \leq 3$. We collect these results in the following proposition. 

\begin{proposition}\label{prop: known_gonalities} \phantom{a}
\begin{enumerate}
    \item We have $\gamma_{\Qbar}(X_0^D(N)) = 1$ (i.e., $g(X_0^D(N)) = 0$) if and only if 
    \[ (D,N) \in \{ (6,1), (10,1), (22,1) \}. \]
    \item We have $g(X_0^D(N)) = 1$ if and only if $(D,N)$ belongs to
    \[ \{(6,5), (6,7), (6,13), (10,3), (10,7), (14,1), (15,1), (21,1), (33,1), (34,1), (46,1)   \}. \]
    \item We have $g(X_0^D(N)) \geq 2$ and $\gamma_\Q(X_0^D(N)) = 2$ if and only if  $(D,N)$ belongs to
    \begin{align*}
       \{ &(6,11), (6,19), (6,29), (6,31), (6,37), (10,11),(10,23), (14,5), (15,2), (22,3), (22,5), \\
& (26,1), (35,1), (38,1), (39,1), (39,2) , (51,1), (55,1), (58,1), (62,1), (69,1), (74,1), (86,1), \\
& (87,1), (94,1), (95,1), (111,1), (119,1), (134,1), (146,1), (159,1), (194,1), (206,1) \}.
\end{align*}
    \item We have $g(X_0^D(N)) \geq 2$ and $\gamma_{\Qbar}(X_0^D(N)) = 2 < \gamma_\Q$ if and only if
        \[  (D,N) \in \{(6,17), (10,13), (10,19), (14,3), (15,4), (21,2), (26,3), (57,1), (82,1), (93,1)\}. \]
    \item We have $\gamma_{\Qbar}(X_0^D(N)) = 3$ if and only if
     \[ (D,N) \in \{(106,1),(118,1)\} . \]
\end{enumerate}
\end{proposition}
\begin{proof}
Parts~(1) and (2) follow from standard genus bounds and computations. Parts~(3) and (4) are work of Ogg \cite{Ogg83} and Guo--Yang \cite{GY17}. Part~(5) follows from \cite[Theorem~7.4]{PS25} and the prior parts. 
\end{proof}

\subsection{Proof of \cref{theorem: tetragonal_introduction}}

A main tool for us is the following lower bound on the geometric gonality of $X_0^D(N)$. This is a special case of a more general theorem of Abramovich applying to all Shimura curves, where we have improved the constant appearing using the best known result on Selberg's eigenvalue conjecture \cite[p. 176]{K03}. 

\begin{theorem}{\cite[Theorem~1.1]{Abramovich96}}\label{theorem: Abr}
For the Shimura curve $X_0^D(N)$, we have
\[
\gamma_{\C}(X_0^D(N)) \ge \frac{975}{8192}\big(g(X_0^D(N))-1\big). 
\]
\end{theorem}

If $X_0^D(N)$ is geometrically tetragonal, then 
\[
\gamma_\C(X_0^D(N)) \leq 4.
\]
It follows from \cref{theorem: Abr} that $g(X_0^D(N)) \leq 34$. We have the following lower bound on this genus in terms of the product $DN$.

\begin{lemma}\cite[Lemma 10.6]{Saia24}\label{Lemma: Saia_genus_bound}
Let $\gamma$ denote the Euler--Mascheroni constant. For $D>1$ the discriminant of an indefinite quaternion algebra over $\Q$ and $N$ a positive integer coprime to $D$, we have
\[
g(X_0^D(N)) > 1 + \frac{DN}{12}\left( \frac{1}{e^\gamma \log\log(DN) + \frac{3}{\log\log{6}}} \right)- \frac{7\sqrt{DN}}{3}.
\]
\end{lemma}

From \cref{Lemma: Saia_genus_bound}, we find that if $DN > 77416$, then $g(X_0^D(N)) > 34$ and, hence, $\gamma_\Q(X_0^D(N)) > 4$, so $X_0^D(N)$ is not tetragonal over $\Q$. For each pair $(D,N)$, with $D>1$ an indefinite quaternion discriminant and $N$ a positive integer coprime to $D$, satisfying $DN \leq 77416$, we check the inequality $g(X_0^D(N)) \leq 34$ required for $X_0^D(N)$ to be tetragonal by \cref{theorem: Abr}. If this inequality is satisfied and $g(X_0^D(N)) \geq 2$, then we call such a pair a \emph{tetragonal candidate}. All $516$ tetragonal candidates are computed in \texttt{tetragonal{\_}checks.m} and listed in the file \texttt{tetragonal{\_}candidates.m} in \cite{MPSSRep}. We now narrow down these candidates.
\begin{proposition}\label{prop: tetragonal_sieving}
\phantom{a}
\begin{enumerate}
    \item The curve $X_0^D(N)$ is geometrically tetragonal for the $133$ pairs $(D,N)$ listed in \cref{table: tetragonal_AL2}, for the $24$ pairs listed in \cref{table: tetragonal_AL4}, and for the $4$ pairs listed in \cref{table: geom_tetragonal_by_Polizzi}. 
    \item The curve $X_0^D(N)$ is tetragonal over $\Q$ for the $60$ pairs in the set 
    \begin{align*} 
    \{ &( 6, 11 ),
    ( 6, 17 ),
    ( 6, 19 ),
    ( 6, 23 ),
    ( 6, 25 ),
    ( 6, 29 ),
     ( 6, 31 ),
    ( 6, 37 ),
    ( 6, 41 ),
    ( 6, 71 ),
    (10,9),
     \\
    &( 10, 11 ),
    ( 10, 13 ),
    ( 10, 17 ),
    ( 10, 23 ),
    ( 10, 29 ),
    ( 14, 3 ),
    ( 14, 5 ),
    ( 14, 13 ), 
    ( 14, 19 ),
    ( 15, 2 ), \\
    & ( 15, 4 ),
    ( 15, 7 ),
    ( 15, 13 ),
    ( 15, 17 ),
     ( 21, 2 ),
    ( 21, 5 ),
    ( 21, 11 ),
    ( 22, 3 ),
    ( 22, 5 ), 
     ( 22, 7 ),\\
    & ( 22, 17 ),
    ( 26, 1 ),
     ( 33, 2 ), 
    ( 33, 7 ),
    ( 35, 1 ),
     ( 35, 2 ),
    ( 38, 1 ),
    ( 39, 1 ),
    ( 39, 2 ),
    ( 39, 4 ), 
      ( 46, 5 ), \\
    &  ( 51, 1 ), 
    ( 51, 2),
    ( 55, 1 ),
    ( 58, 1 ),
    ( 62, 1 ),
    ( 69, 1 ),
    ( 74, 1 ), 
    ( 86, 1 ),
    ( 87, 1 ),
    ( 94, 1 ),
    ( 95, 1 ), \\
    & ( 111, 1 ),
    ( 119, 1 ),
    ( 134, 1 ),
    ( 146, 1 ),
    ( 159, 1 ),
    ( 194, 1 ),
    ( 206, 1 ) \}
    \end{align*}
    and for the $81$ pairs listed in \cref{table: tetragonal_by_CM}.
    
    \item The curve $X_0^D(N)$ is not geometrically tetragonal (and, \emph{a fortiori}, not tetragonal over $\Q$) for the $322$ pairs $(D,N)$ listed in \cref{table: non-tetragonal_CS2}, for the pair $(14,31)$, and for the $3$ pairs in the set
    \[ \{ (427,1), (445, 1), (505,1)\}. \]
    
     \item The curve $X_0^D(N)$ is not tetragonal over $\Q$ for the $8$ pairs $(D,N)$ in the set
    \[ \{     ( 51, 7 ),
    ( 62, 5 ),
    ( 65, 2 ),
    ( 205, 1 ),
    ( 217, 1 ),
    ( 267, 1 ),
    ( 301, 1 ),
    ( 478, 1 )\}. \]
\end{enumerate}
\end{proposition}
\begin{proof}
The code for all computations described in this proof can be found in the file \texttt{tetragonal{\_}sieving.m} in \cite{MPSSRep}.

\begin{enumerate}
    \item We witness the curves in \cref{table: tetragonal_AL2} and \cref{table: tetragonal_AL4} as being geometrically tetragonal via the existence of an Atkin--Lehner quotient map $X_0^D(N) \to X_0^D(N)/W$ with $W \leq W_0(D,N)$ of order $4$ and with $g(X_0^D(N)/W) = 0$, or via the existence of such a map with $|W| = 2$ and with $g(X_0^D(N)/W) \leq 2$. All curves with at least one of these types of quotient maps appear in one of these tables, though we do not repeat pairs to be exhaustive in listing all such maps. For example, we only list $(6,25)$ in \cref{table: tetragonal_AL2}, but it also has an Atkin--Lehner subgroup of order~$4$ yielding a conic quotient curve. 
    
    We find that the four pairs listed in \cref{table: geom_tetragonal_by_Polizzi} are geometrically tetragonal via the existence of a sequence of covers 
    \[ X_0^D(N) \to X_0^D(N) / \langle w_{m_1} \rangle \to X_0^D(N) / \langle w_{m_1}, w_{m_2} \rangle, \]
    for $m_1 \neq m_2$ Hall divisors of $DN$, such that the bottom curve has genus~$2$ and the intermediate quotient has genus~$3$. It then follows from \cite[Theorem~3.4]{Pol06} that the intermediate quotient is geometrically hyperelliptic (and bielliptic), and so $X_0^D(N)$ is geometrically tetragonal. We list such values of $m_1$ and $m_2$ for each pair in \cref{table: geom_tetragonal_by_Polizzi}.

    \item If $(D,N)$ satisfies at least one of the following conditions, then the curve $X_0^D(N)$ is tetragonal over~$\Q$:
    \begin{enumerate}
        \item there is a subgroup $W \leq W_0(D,N)$ of order~$4$ so that $X_0^D(N)/W$ is isomorphic to $\mathbb{P}^1_\Q$, i.e., has genus~$0$ and a rational point;
        \item $X_0^D(N)$ is bielliptic over~$\Q$;
        \item $X_0^D(N)$ has a degree~$2$ map to a curve which is sub-hyperelliptic over~$\Q$. 
    \end{enumerate}
    If $X_0^D(N)$ is geometrically hyperelliptic (and hence, \emph{a fortiori}, if it is hyperelliptic over~$\Q$), then the hyperelliptic involution is an Atkin--Lehner involution~$w_m$ \cite[p.~301]{Ogg83}. Suppose that $X_0^D(N)$ is hyperelliptic over~$\Q$ with hyperelliptic involution~$w_m$, and take some involution $w_{m'} \neq w_m$. The quotient $X_0^D(N)/\langle w_{m}, w_{m'} \rangle$, like $X_0^D(N)/\langle w_m \rangle$, must be isomorphic to $\mathbb{P}^1_\Q$; so in this case $(D,N)$ satisfies the item~(a) above and is tetragonal. By Part~(3) of \cref{prop: known_gonalities}, this applies to the $33$ pairs in the following set
    \begin{align*}  
    \{ &( 6, 11 ),
    ( 6, 19 ),
    ( 6, 29 ),
    ( 6, 31 ),
    ( 6, 37 ),
    ( 10, 11 ),
    ( 10, 23 ),
    ( 14, 5 ),
    ( 15, 2 ),
    ( 22, 3 ),
    ( 22, 5 ), \\
    &( 26, 1 ),
    ( 35, 1 ),
    ( 38, 1 ),
    ( 39, 1 ),
    ( 39, 2 ),
    ( 51, 1 ),
    ( 55, 1 ),
    ( 58, 1 ),
    ( 62, 1 ),
    ( 69, 1 ),
    ( 74, 1 ), \\
    &( 86, 1 ),
    ( 87, 1 ),
    ( 94, 1 ),
    ( 95, 1 ),
    ( 111, 1 ),
    ( 119, 1 ),
    ( 134, 1 ),
    ( 146, 1 ),
    ( 159, 1 ),
    ( 194, 1 ),
    ( 206, 1 ) \}. 
    \end{align*}
    
    From \cite{PS25}, the curve $X_0^D(N)$ is bielliptic over $\Q$, satisfying the item~(b), for the following $23$ pairs $(D,N)$:
        \begin{align*} 
        \{&( 6, 17 ),
    ( 6, 23 ),
    ( 6, 25 ),
    ( 6, 41 ),
    ( 6, 71 ),
    ( 10, 13 ),
    ( 10, 17 ),
    ( 10, 29 ),
    ( 14, 3 ),
    ( 14, 13 ), \\
    &( 14, 19 ),
    ( 15, 7 ),
    ( 15, 13 ),
    ( 15, 17 ),
    ( 21, 2 ),
    ( 21, 5 ),
    ( 21, 11 ),
    ( 22, 7 ),
    ( 22, 17 ),
    ( 33, 2 ), \\
    &( 33, 7 ),
    ( 35, 2 ),
    ( 46, 5 )\}.
    \end{align*}
      
    The pairs $(15,4)$ and $(39,4)$ satisfy the item~(c), as the corresponding curves are double covers of the curves $X_0^{15}(2)$ and $X_0^{39}(2)$, respectively, which are hyperelliptic over~$\Q$. The pair $(51,2)$ also satisfies the item~(c); its quotient $X_0^{51}(2)/\langle w_{51} \rangle$ is bielliptic of genus~$2$ over~$\Q$, with equation given in \cite{PS25b}. The curve $X_0^{10}(9)$ is tetragonal by the item~(a): its genus~$0$ quotient by the subgroup $\langle w_2, w_{45}\rangle$ has a rational point by \cite[Proposition~4.25]{NR15}.
    
    For the remaining pairs which we found to be geometrically tetragonal in Part~(1) via Atkin--Lehner quotient maps, we attempt to show that $X_0^D(N)$ is tetragonal over~$\Q$ by showing that either
    \begin{itemize}
        \item there is an Atkin--Lehner quotient map $X_0^D(N) \to X_0^D(N)/\langle w_m \rangle$ of  degree~$2$ to a sub-hyperelliptic curve with a rational CM point, or 
        \item there is an Atkin--Lehner quotient map $X_0^D(N) \to X_0^D(N)/\langle w_{m_1}, w_{m_2} \rangle$ of degree~$4$ with the latter quotient having genus~$0$ and a rational CM point, such that it is isomorphic to $\mathbb{P}^1_{\Q}$.
    \end{itemize}
    In both cases, we prove the existence of rational CM points on such quotients by using the work of \cite{GR06} and \cite{Saia24} to determine all quadratic CM points on $X_0^D(N)$, and using \cite[Corollary 5.14]{GR06} to determine such points whose image on a quotient of the form $X_0^D(N)/\langle w_m \rangle$ has residue field $\Q$. All $83$ pairs handled in this manner, along with an Atkin--Lehner involution whose corresponding quotient provably has a rational CM point and is used in the argument, are listed in \cref{table: tetragonal_by_CM}.
    
    \item For the $322$ pairs listed in \cref{table: non-tetragonal_CS2}, we apply \cref{lemma: CS_tetragonal_with_involution} (to the base change of $X_0^D(N)$ to $\overline{\mathbb{Q}}$) by finding a Hall divisor $m$ of $DN$ such that \cref{CS_tetragonal} holds with $\sigma = w_m$ and such that $X_0^D(N)/\langle w_m \rangle$ is not geometrically sub-hyperelliptic. Such a quotient must have genus $g_m > 2$, as otherwise we would have found these curves $X_0^D(N)$ to be geometrically tetragonal in Part~(1), and we prove such quotients are not geometrically hyperelliptic by one of the following means:
    \begin{enumerate}
        \item We have knowledge that this curve has no non-Atkin--Lehner automorphisms by the results of \cref{section: automorphisms}, and so it suffices to check for Atkin--Lehner quotients $X_0^D(N)/\langle w_m,w_{m'}\rangle$ with $m'>1$ and $m' \neq m$ of genus~$0$. 
        \item We apply \cref{lemma: CS_geom_hyperelliptic_bound} to $X = X_0^D(N)/\langle w_m \rangle$ with $\sigma = w_{m'}$ for some Hall divisor $m' > 1$ of $DN$ with $m' \neq m$.
        \item We use \cref{proposition: not_hyperelliptic}, showing that $g_m \neq \frac{g}{2}$ or that $g_m \not\in\left\{\frac{g-1}{2}, \frac{g + 1}{2}\right\}$, for $g=g(X_0^D(N))$ and $g_m = g\left(X_0^D(N)/\langle w_m \rangle\right)$, according to whether $g$ is even or odd.
        \item We use finite field point counts to show that the inequality in \cref{lemma: hyperelliptic_bielliptic_point_count_bounds}, Part~(1), is not satisfied.
    \end{enumerate}
    
    For the pair $(14,31)$, we argue similarly using the Castelnuovo--Severi inequality with a larger Atkin--Lehner subgroup: $X_0^{14}(31)$ has a degree~$4$ map to the genus~$1$ curve $X_0^{14}(31)/\langle w_{2}, w_{217} \rangle$. If $X_0^{14}(31)$ were geometrically tetragonal, then a degree~$4$ map to a conic would be independent of the quotient map to~$X_0^{14}(31)/\langle w_{2}, w_{217} \rangle$ unless both maps factored through a geometrically hyperelliptic curve $X_0^D(N)/\langle w_m \rangle$ for some $w_m \in \langle w_{2}, w_{217} \rangle$. We rule out this possibility similarly to the way we did in the previous paragraph. The Castelnuovo--Severi inequality then implies
    \[ 17 = g(X_0^{14}(31)) \leq 4 \cdot 0 + 4 \cdot 1 + 3 \cdot 3 = 13, \]
    giving a contradiction.
    
    We now handle the last $3$ mentioned pairs using \cref{prop: NO_corollary}. For each of these three curves $X = X_0^D(N)$, there exists a quotient $Y = Y_0^D(N)/\langle w_m \rangle$ of genus $g_Y \geq 10$ such that $Y$ has a rational CM point. By \cref{lemma: CS_unique_hyperelliptic} and \cref{prop: NO_corollary}, we then have
    \[ \gamma_\C(X) \geq \gamma_\C(Y) = \gamma_\Q(Y). \]
    We then prove that $X$ is not geometrically tetragonal by exhibiting such a curve $Y$ for which point count computations and \cref{lemma: gonality_bound_by_point_counts} prove that $\gamma_\Q(Y) > 4$. For example, the genus $15$ quotient $Y = X_0^{505}(1)/\langle w_{505}\rangle$ of $X = X_0^{505}(1)$ has a rational $-3$-CM point and $Y(\F_4) = 24 > 20$, so $\gamma_\C(Y) > 4$. 
    
    \item The $8$ referenced curves $X_0^D(N)$ are proven to not be tetragonal over $\Q$ using point count computations and \cref{lemma: gonality_bound_by_point_counts}. 
\end{enumerate}
\end{proof}

\begin{proposition}\label{prop: tetragonal_over_Q_iff}
\phantom{a}
\begin{enumerate}
    \item For the $14$ triples $(D,N,W)$ in the set 
\begin{align*}
    \{&(6,55,\langle w_{10}, w_{66} \rangle), (6,77, \langle w_{21},w_{66} \rangle), (6,83,\langle w_{6}, w_{166} \rangle), (10,33,\langle w_{30}, w_{33} \rangle), \\
    &(14,15,\langle w_{14}, w_{30} \rangle), (14, 23,\langle w_{14}, w_{46} \rangle), (21,13,\langle w_{21}, w_{39} \rangle), (22,19,\langle w_{22}, w_{38} \rangle), \\
    &(51,5,\langle w_{5}, w_{51} \rangle), (55,2,\langle w_{2}, w_{55} \rangle), (86,3,\langle w_{6}, w_{86} \rangle), (87,2,\langle w_{2}, w_{87} \rangle), \\
    & (95, 2, \langle w_{10}, w_{38} \rangle), (111,2,\langle w_{2}, w_{111} \rangle)\},
\end{align*}
we have that $X_0^D(N)$ is tetragonal over $\Q$ if and only if $\left(X_0^D(N)/W\right)(\Q) \neq \varnothing$. 

\item The curve $X_0^{46}(7)$ is tetragonal over $\Q$ if and only if $X_0^{46}(7)/\langle w_{161} \rangle$ is hyperelliptic over $\Q$. 
\end{enumerate}
\end{proposition}
\begin{proof}
\phantom{a}
\begin{enumerate}
    \item For each $(D,N,W)$ in the given set, we know $X_0^D(N)$ is geometrically tetragonal via the natural degree~$4$ map to the genus~$0$ curve $X_0^D(N)/W$. We also have $g(X_0^D(N)) \geq 10$, and we find that $X_0^D(N)$ is neither geometrically hyperelliptic nor geometrically bielliptic (either using finite field point count computations and \cref{lemma: hyperelliptic_bielliptic_point_count_bounds} or using knowledge that all involutions are Atkin--Lehner from the results of \cref{section: automorphisms}). By \cref{corollary: CS_not_tetragonal}, we then know that $X_0^D(N)$ is tetragonal over~$\Q$ if and only if the known geometrically tetragonal map is tetragonal over~$\Q$, i.e., if $X_0^D(N)/W \cong \P^1_\Q$. 

    \item The genus~$3$ curve $Y = X_0^{46}(7)/\langle w_{161} \rangle$ is a degree~$2$ cover of the genus~$2$ curve $X_0^{46}(7)/\langle w_7, w_{23} \rangle$, and, hence, is geometrically hyperelliptic by \cite[Theorem~3.4]{Pol06}. We recognize that $X = X_0^{46}(7)$ is geometrically tetragonal via the composition of the natural quotient map $X \to Y$ with the hyperelliptic quotient map $Y \to C$ to a conic over~$\Q$, and note in particular that this composition is defined over~$\mathbb{Q}$. As $X$ has genus $13 \geq 10$ and is proven to be neither geometrically hyperelliptic nor geometrically bielliptic by either of the same checks referred to in the proof of Part~(1), this composition is the unique geometrically tetragonal map on~$X$ by \cref{lemma: CS_unique_tetragonal}. Thus, $X$ is tetragonal over~$\Q$ if and only if the quotient of~$Y$ by its hyperelliptic involution has a rational point (as in \cref{corollary: CS_not_tetragonal}). \qedhere
\end{enumerate}
\end{proof}

\begin{corollary}\label{not_tetragonal_cor}
For the $9$ pairs $(D,N)$ in the set 
\begin{align*}
    \left \{(6,55), (6,77), (6,83), (10,33), (14,15), (14,23), (21, 13), (22, 19), (86,3) \right\},
\end{align*}
we have that $X_0^D(N)$ is not tetragonal over $\Q$.
\end{corollary}
\begin{proof}
For each pair $(D,N)$ listed in the statement of the corollary, we know from  Part~(1) of \cref{prop: tetragonal_over_Q_iff} that $X_0^D(N)$ is tetragonal over $\Q$ if and only if the quotient $X_0^D(N)/W$ specified therein has a rational point. The claim then follows from \cite[Corollary~1.3]{PS25b}. (More specifically: using \cite[Algorithm~3.3]{PS25b}, one can compute that there is at least one prime $p \mid D$ such that $X_0^D(N)/W$ has no $\Q_p$ points, giving the claim.) 
\end{proof}

From the $516$ tetragonal candidates, \cref{prop: tetragonal_sieving} narrows us down to only $29$ pairs for which we remain unsure of whether $X_0^D(N)$ is geometrically tetragonal. These are listed in \cref{table: remaining_geom_tetragonal_candidates}. Of course, if $X_0^D(N)$ is tetragonal over~$\Q$, then it is also geometrically tetragonal, and so it must appear in \cref{table: remaining_geom_tetragonal_candidates} or in \cref{table: tetragonal_AL2}, \cref{table: tetragonal_AL4}, or \cref{table: geom_tetragonal_by_Polizzi}. Of all such pairs, following \cref{not_tetragonal_cor}, there remain $32$ pairs for which we remain unsure of whether $X_0^D(N)$ is tetragonal over~$\Q$. These pairs are listed in \cref{table: remaining_tetragonal_candidates}. Note that this latter table does not include all $32$ of the pairs in \cref{table: remaining_geom_tetragonal_candidates}; there are $8$ curves which we prove are not tetragonal over~$\Q$ using finite field point counts but for which we remain unsure of whether they are geometrically tetragonal, and likewise there are $11$ curves $X$ which are geometrically tetragonal but for which we remain unsure of whether $X$ is tetragonal over~$\Q$. This completes the proof of \cref{theorem: tetragonal_introduction}.

{
\begin{longtable}{|c|c||c|c||c|c||c|c|}
\caption{All $29$ pairs $(D,N)$ for which we remain unsure of whether $X_0^D(N)$ is geometrically tetragonal following \cref{theorem: tetragonal_introduction}. Pairs for which we have proved that $X_0^D(N)$ is not tetragonal over $\Q$ are written in bold.}\label{table: remaining_geom_tetragonal_candidates} \\  \hline 
$(D,N)$ & $g$ & $(D,N)$ & $g$ & $(D,N)$ & $g$ & $(D,N)$ & $g$ \\ \hline \hline
$ (6 , 73 )$ & $ 11 $
 & 
$ (10 , 27 )$ & $ 13 $
 & 
$ (15 , 16 )$ & $ 17 $
 & 
$ (22 , 9 )$ & $ 11 $
\\ \hline
$ (33 , 4 )$ & $ 11 $
 & 
$ (34 , 13 )$ & $ 17 $
 &
$ (38 , 5 )$ & $ 9 $
 & 
$ (46 , 9 )$ & $ 23 $
\\ \hline 
$\mathbf{ (51 , 7 )}$ & $\mathbf{ 21 }$
&
$ (55 , 3 )$ & $ 13 $
 & 
$\mathbf{ (62 , 5 )}$ & $\mathbf{ 15 }$
 & 
$\mathbf{ (65 , 2 )}$ & $\mathbf{ 13 }$
\\ \hline
$ (69 , 2 )$ & $ 11 $
 & 
$ (94 , 3 )$ & $ 15 $
 & 
$ (133 , 1 )$ & $ 9 $
 &
$ (177 , 1 )$ & $ 9 $
\\ \hline 
$ (187 , 1 )$ & $ 13 $
 & 
$ \mathbf{(205 , 1 )}$ & $\mathbf{ 13 }$
 &
$ (213 , 1 )$ & $ 11 $
 & 
$ (214 , 1 )$ & $ 8 $
\\ \hline 
$ \mathbf{(217 , 1 )}$ & $ \mathbf{15} $
&
$ (226 , 1 )$ & $ 9 $
 & 
$ (262 , 1 )$ & $ 10 $
 & 
$\mathbf{ (267 , 1 )}$ & $\mathbf{ 15} $
\\ \hline
$ (298 , 1 )$ & $ 12 $
 & 
$\mathbf{ (301 , 1 )}$ & $\mathbf{ 21 }$
 & 
$ (358 , 1 )$ & $ 14 $
 &
$ (382 , 1 )$ & $ 15 $
\\ \hline 
$ \mathbf{(478 , 1 )}$ & $ \mathbf{19} $ & & & & & & \\ \hline
\end{longtable}
}

{
\begin{longtable}{|c|c||c|c||c|c||c|c|}
\caption{All $32$ pairs $(D,N)$ for which we remain unsure of whether $X_0^D(N)$ is tetragonal over $\Q$ following \cref{theorem: tetragonal_introduction}. Pairs for which we have proved that $X_0^D(N)$ is geometrically tetragonal are written in bold.}\label{table: remaining_tetragonal_candidates} \\ \hline 
$(D,N)$ & $g$ & $(D,N)$ & $g$ & $(D,N)$ & $g$ & $(D,N)$ & $g$ \\ \hline 
$ \mathbf{(
6 , 49 )}$ & $\mathbf{ 9 }$
 & 
$ (
6 , 73 )$ & $ 11 $
 & 
$ (
10 , 27 )$ & $ 13 $

 &
 $\mathbf{ (
14 , 9 )}$ & $\mathbf{ 7 }$
\\ \hline 
$ \mathbf{(
15 , 8 )}$ & $\mathbf{ 9 }$
 & 
$ (
15 , 16 )$ & $ 17 $
 &
$ \mathbf{(
21 , 4 )}$ & $\mathbf{ 7 }$
 & 
$ \mathbf{(
21 , 8 )}$ & $\mathbf{ 13 }$
\\ \hline 
$ (
22 , 9 )$ & $ 11 $
&
$ (
33 , 4 )$ & $ 11 $
 & 
$ (
34 , 13 )$ & $ 17 $
 & 
$ (
38 , 5 )$ & $ 9 $
\\ \hline
$ \mathbf{(
46 , 7 )}$ & $\mathbf{ 13 }$
 & 
$ (
46 , 9 )$ & $ 23 $
 & 
$\mathbf{ (
51 , 5 )}$ & $\mathbf{ 17 }$
 &
$ \mathbf{(
55 , 2 )}$ & $\mathbf{ 11 }$
 \\ \hline
$ (
55 , 3 )$ & $ 13 $
 & 
$ (
69 , 2 )$ & $ 11 $
 &
$\mathbf{ (
87 , 2 )}$ & $\mathbf{ 15 }$
 & 
$ (
94 , 3 )$ & $ 15 $
\\ \hline 
$\mathbf{ (
95 , 2 )}$ & $\mathbf{ 19 }$
&
$\mathbf{ (
111 , 2 )}$ & $\mathbf{ 19 }$
 & 
$ (
133 , 1 )$ & $ 9 $
 & 
$ (
177 , 1 )$ & $ 9 $
\\ \hline
$ (
187 , 1 )$ & $ 13 $
 & 
$ (
213 , 1 )$ & $ 11 $
 & 
$ (
214 , 1 )$ & $ 8 $
 &
$ (
226 , 1 )$ & $ 9 $
\\ \hline 
$ (
262 , 1 )$ & $ 10 $
 & 
$ (
298 , 1 )$ & $ 12 $
&
$ (
358 , 1 )$ & $ 14 $ & 
$ (
382 , 1 )$ & $ 15 $ \\ \hline
\end{longtable}
}

The pairs in \cref{table: remaining_geom_tetragonal_candidates} and \cref{table: remaining_tetragonal_candidates} appear to require further techniques to handle, and we leave them to future work. In particular, those pairs in \cref{table: remaining_geom_tetragonal_candidates} with $N$ not squarefree are not amenable to our techniques for determining the automorphism group from \cref{section: automorphisms}, and those pairs with genus $g \leq 9$ are not amenable to \cref{corollary: CS_not_tetragonal}.


\section{Point count records and maximal curves}\label{section: maximal_curves}

In this section, we display tables with our best results on the number of points for a curve of given genus over a specified finite field among Atkin--Lehner quotients of Shimura curves $X_0^D(N)$.

\cref{tabrecord} collects $116$ improvements with respect to data collected on \cite{ManyPoints}. These improvements have been shared with the authors of the referenced site and are included in their database at the time of writing this work. They are all (non-trivial) Atkin--Lehner quotients of Shimura curves $X_0^D(N)$.

{
\begin{longtable}{|R|R|R|R|R||R|R|R|R|R|}

\caption{The 116 improved bounds for $|X(\F_{p^k})|$ 
among curves $X$ of the form $X_0^D(N)/\langle w_{m_1},\ldots,w_{m_r}\rangle$ for $DN \le 10000$, $D>1$, genus $g\le 50$, primes $p\le 19$, and $k\le 5$.}\label{tabrecord} \\
\hline 
\rule{0pt}{2.3ex} g & p^k & (D,N) & m_1,\ldots,m_r & |X(\F_{p^k})| & g & p^k & (D,N) & m_1,\ldots,m_r & |X(\F_{p^k})|\\
\hline \hline 
\rule{0pt}{2.3ex} 
8 & 3^2 & (14,95) & 5,266 & 44 & 
8 & 13^2 & (6,209 ) & 57,66 & 370 \\ 
\hline 
\rule{0pt}{2.3ex} 
8 & 17^5 & (546,5 ) & 2,91,39 & 1436014 &
10 & 7^2 & (10,291) & 6,5,97 & 170\\ 
\hline 
\rule{0pt}{2.3ex} 
10 & 7^5 & (38,15) & 3,19 & 18850 &
10 & 17^5 & (14,141 ) & 2,7,47 & 1441950 \\ 
\hline 
\rule{0pt}{2.3ex} 
10 & 19^2 & (215,4) & 5,43 & 720  & 
11 & 2^2 & (4485,1) & 3,13,115 & 28 \\
\hline 
\rule{0pt}{2.3ex} 
11 & 5^5 & (21,67) & 7,67 & 4000 &
11 & 7^2 & (6474,1)& 3,83,26& 184\\ 
\hline 
\rule{0pt}{2.3ex} 
11 & 7^5 &(4182,1)& 17,6,123 & 18883 &
11 & 19^2 & (39,34 ) & 6,13,17 & 748\\ 
\hline 
\rule{0pt}{2.3ex} 
12 & 7^2 & (26,57)& 57,78& 190 &
12 & 7^5 & (6,475) & 25,57,6 & 19682\\ 
\hline 
\rule{0pt}{2.3ex} 
12 & 11^2 & (15,178 ) & 2,89,15 & 342  &
12 & 13^5 & (6,265) &265,30& 381782\\
\hline 
\rule{0pt}{2.3ex} 
12 & 17^5 & (570,11) & 5,6,209 & 1441318 & 
13 & 7^2 & (46,25) & 2,23 & 186\\
\hline 
\rule{0pt}{2.3ex} 
13 & 7^5 & (82,15) & 41,10 & 19456  &
13 & 11^5 & (39,8 ) & 3 & 168744\\
\hline 
\rule{0pt}{2.3ex} 
13 & 13^2 & (10,99 ) & 10,22 & 456  &
13 & 13^4 & (35,9) & 7 & 32436\\  
\hline 
\rule{0pt}{2.3ex} 
13 & 17^5 & (2415,1 ) & 15,161 & 1444848  & 
14 & 3^2 & (65,14) & 5,91 & 60 \\
\hline 
\rule{0pt}{2.3ex} 
14 & 7^2 & (10,97) & 97 & 198  &
14 & 13^2 & (714,11) & 33,238,1309 & 460\\ 
\hline 
\rule{0pt}{2.3ex} 
14 & 13^5 & (6,265) & 30,159& 384006  &  
14 & 19^5 & (134,15 ) & 2,5,67 & 2514998\\ 
\hline 
\rule{0pt}{2.3ex} 
15 & 13^5 & (15,112)&16,5,21& 386416  & 
15 & 17^2 & (1590,1 ) & 30 & 724 \\ 
\hline 
\rule{0pt}{2.3ex} 
17 & 7^5 & (82,15) & 2,41 & 20106  & 
17 & 13^2 & (10,63) & 10 & 544 \\
\hline 
\rule{0pt}{2.3ex} 
17 & 13^5 & (15,32) & 5 & 386880 &  
18 & 7^2 & (22,43) & 43 & 230\\
\hline 
\rule{0pt}{2.3ex} 
18 & 17^2 & (26,57) & 3,247 & 784  &
19 & 7^5 & (26,27) & 13 & 20520\\ 
\hline 
\rule{0pt}{2.3ex} 
19 & 17^5 & (14,141 ) & 14,94 & 1451936 &
20 & 19^2 & (14,117 ) & 13,63 & 946\\
\hline 
\rule{0pt}{2.3ex} 
21 & 5^2 & (14,277) & 7,554 & 164 & 
21 & 11^2 & (57,14 ) & 399 & 476\\ 
\hline 
\rule{0pt}{2.3ex} 
21 & 13^2 & (10,171) & 10,19 & 596 & 
21 & 17^2 & (2415,2) & 15,14,46 & 848 \\ 
\hline 
\rule{0pt}{2.3ex} 
22 & 2^2 & (15,197) & 3,985 & 45 & 
22 & 3^2 & (215,4) & 215 & 84 \\
\hline 
\rule{0pt}{2.3ex} 
22 & 11^2 & (14,135 ) & 7,270 & 498 &
22 & 13^2 & (57,28) & 19,21 & 594 \\
\hline 
\rule{0pt}{2.3ex} 
22 &  19^2 & (51,26 ) & 39,51 & 1040 &
23 & 2^2 & (39,31) & 1209 & 46 \\
\hline 
\rule{0pt}{2.3ex} 
23 & 11^2 & (35,26) & 455 &508 & 
23 & 13^2 & (22,49) & 22 & 640 \\
\hline 
\rule{0pt}{2.3ex} 
23 & 17^2 & (22,87) & 3,58 & 938  & 
23 & 19^2 & (690,7 ) & 14,483 & 1004\\
\hline 
\rule{0pt}{2.3ex} 
24 & 5^2 & (14,253) & 7, 506 & 190 & 
24 & 17^2 & (2415,2) & 161,69,115 & 910 \\ 
\hline 
\rule{0pt}{2.3ex} 
25 & 13^2 & (10,99) & 22 & 696  &
25 & 19^2 & (34,123 ) & 3,17,41 & 1091\\ 
\hline 
\rule{0pt}{2.3ex} 
26 & 11^2 & (21,118 ) & 14,1239 & 550 & 
26 & 17^2 & (6,539) & 2,147 & 928 \\
\hline 
\rule{0pt}{2.3ex} 
26 & 19^2 & (26,85 ) & 13,34 & 1106 &
27 & 13^2 & (4935,1) &5,987 & 686 \\ 
\hline 
\rule{0pt}{2.3ex} 
27 & 17^2 & (3045,2 ) & 58,435,1015 & 978  & 
28 & 5^2 & (38,27) & 19 & 216 \\
\hline 
\rule{0pt}{2.3ex} 
28 & 11^2 & (35,52 ) & 5,91 & 570 & 
28 & 13^2 & (15,164) & 15,41 & 756\\ 
\hline 
\rule{0pt}{2.3ex} 
28 & 19^2 & (33,58 ) & 11,29 & 1152  & 
29 & 3^2 & (35,89) & 5,623 & 108\\
\hline 
\rule{0pt}{2.3ex} 
29 & 5^2 & (623,3) & 3,623 & 208  & 
29 & 11^2 & (65,14 ) & 455 & 572\\ 
\hline 
\rule{0pt}{2.3ex} 
29 & 13^2 & (33,56) & 11,21 & 704  &
29 & 17^2 & (15,77 ) & 33 & 1020\\ 
\hline 
\rule{0pt}{2.3ex} 
29 & 19^2 & (10,117) & 130 & 1216  &
30 & 2^2 & (1169,1) & 1169 & 55\\
\hline 
\rule{0pt}{2.3ex} 
30 & 3^2 & (391,4) & 17,23 & 109  & 
30 & 13^2 & (118,9) & 118 & 746 \\
\hline 
\rule{0pt}{2.3ex} 
30 & 17^2 & (177,8 ) & 3,59 & 1050 &
31 & 11^2 & (14,345 ) & 5,23,14 & 684\\
\hline 
\rule{0pt}{2.3ex} 
31 & 13^2 & (15,62) & 93 & 780  &
31 & 19^2 & (1155,2 ) & 7,11 & 1272\\ 
\hline 
\rule{0pt}{2.3ex} 
32 & 3^2 & (1055,2) & 2,1055 & 113  &
32 & 5^2 & ( 77,37 ) & 7,407 & 221 \\ 
\hline 
\rule{0pt}{2.3ex} 
32 & 11^2 & (15,178 ) & 15,89 & 628 & 
32 & 19^2 & (33,62 ) & 11,93 & 1278\\
\hline 
\rule{0pt}{2.3ex} 
33 & 5^2 & ( 57,31 ) & 1767 & 236  &
33 & 13^2 & (1055,2) & 422,10 & 773\\ 
\hline 
\rule{0pt}{2.3ex} 
33 & 17^2 & (1055,2 ) & 10,422 & 1103  & 
34 & 3^2 & (871,2) & 2,871 & 121\\
\hline 
\rule{0pt}{2.3ex} 
34 & 13^2 & (10,297) & 22,297 & 810  & 
34 & 19^2 & (134,9 ) & 67 & 1470\\ 
\hline 
\rule{0pt}{2.3ex} 
36 & 5^2 & (1343 , 1) & 1343 & 250 & 
36 & 11^2 & (142, 9) & 142 & 698 \\
\hline 
\rule{0pt}{2.3ex} 
37 & 5^2 & ( 327, 4) & 327 & 264 & 
37 & 13^2 & (145,23) & 5,667 & 846 \\
\hline 
\rule{0pt}{2.3ex} 
38 & 2^2 & (15,611) & 3,13,235 & 65 &
38 & 5^2 & (3143,1) & 7,449 & 257 \\ 
\hline 
\rule{0pt}{2.3ex} 
39 & 2^2 & (65,31) & 2015 & 68 & 
39 & 5^2 & (3927,2) & 2,231 ,17 & 266 \\
\hline 
\rule{0pt}{2.3ex} 
40 & 5^2 & ( 287, 4) & 287 & 282 & 
40 & 11^2 & (91,20 ) & 13,35 & 762\\
\hline 
\rule{0pt}{2.3ex} 
40 & 19^2 & (10,351 ) & 270,351 & 1488 &
41 & 5^2 & ( 143, 6) & 143 & 280 \\
\hline 
\rule{0pt}{2.3ex} 
42 & 3^2 & (58,97) & 2,2813 & 134 &
43 & 2^2 & (15,301) & 5, 903 & 74 \\ 
\hline 
\rule{0pt}{2.3ex} 
43 & 5^2 & ( 69,22 ) & 759 & 284 & 
43 & 19^2 & (21,230) & 15,161,230 & 1524 \\
\hline 
\rule{0pt}{2.3ex} 
44 & 2^2 & (1589,1) & 1589 & 71 &
44 & 3^2 & (185,19) & 5,703 & 140 \\
\hline 
\rule{0pt}{2.3ex} 
45 & 5^2 & ( 623, 2) & 623 & 308 &
45 & 7^2 & (1751,1) & 1751 & 456 \\
\hline 
\rule{0pt}{2.3ex} 
46 & 7^2 & (695,2) & 695 & 472 &
47 & 2^2 & (7215,1) & 5,1443 & 78 \\
\hline 
\rule{0pt}{2.3ex} 
47 & 7^2 & (15,146) & 1095 & 472 &
47 & 19^2 & (21,62) & 93 & 1624 \\
\hline 
\rule{0pt}{2.3ex} 
48 & 5^2 & (1623,2) & 2,1623& 307 & 
48 & 19^2 & (51,52 ) & 39,51 &  1700 \\
\hline 
\rule{0pt}{2.3ex} 
49 & 17^2 & (35,24 ) & 21 & 1488 &
50 & 5^2 & (57,62) & 2,1767 & 330\\
 \hline 
\end{longtable}
}

In Tables \ref{tabMax2}, \ref{tabMax34}, and \ref{tabMax>4}, we list Atkin--Lehner quotients $X_0^D(N)/\langle w_{d_1},\ldots,w_{d_k}\rangle$ with genus $g\ge 2$ achieving the maximal number of points for a certain pair (genus, finite field). 
The curves are grouped by their real Weil polynomial $h_W(x)$ which is displayed as well.

We write in bold curves that belong to classes sharing a common real Weil polynomial that does not occur among those previously encountered in this paper or in \cite[Appendix C, Tables 11, 12, and 13]{DLMS26}. This implies the existence of a previously unknown maximal isomorphism classes.

We next explain the notation of the tables which we use to save space. Let
\[
DN = \prod_i p_i^{v_i} = \prod_i q_i
\]
be the prime factorization of $DN$, with $p_1<p_2< \ldots$ and $q_i = p_i^{v_i}$.
Then, each Hall divisor $d_h$ is uniquely the product of certain $q_{j}$'s and, in the tables, we denote $X_0^D(N)/\langle w_{d_1},\ldots,w_{d_k}\rangle$ by 
\[
(D,N)\, \{ i_{1,1},\ldots, i_{1,m_1};\ldots \ldots \ldots; i_{k,1},\ldots, i_{k,m_k}\},
\]
where
\[
d_h=\prod_{j=1}^{m_h} q_{i_{h,j}},\qquad\text{for }h=1,\ldots, k.
\]

The top curve $X_0^D(N)$ is denoted by $(D,N)\, \{ \}$ in this notation. For a non-trivial example, the entry ``$(10,27)\{2;1,3\}$" in the table below, corresponding to the third curve in the second class of the row $g=2$ and $q=7$, denotes the quotient of the curve $X_0^{10}(27)$ by $\langle w_{27},w_{10}\rangle$, where $d_1=27=p_2^3$ and $d_2=10=p_1 p_3$ since, in this case, $p_1=2$, $p_2=3$, and $p_3=5$. 

Another example is the entry ``$(462,5)\{2;3;1,5;4,5\}$" in the table below, corresponding to the first curve in the row $g=2$ and $q=17^3$, that denotes the quotient of the curve $X_0^{462}(5)$ by $\langle w_3,w_5,w_{22},w_{77}\rangle$, where $d_1=3=p_2$ and $d_2=5=p_3$ and $d_3=22=p_1 p_5$ and $d_4=77=p_4 p_5$ since, in this case, $p_1=2$, $p_2=3$, $p_3=5$, $p_4=7$, and $p_5=11$.

\vspace{1cm}

{\tiny

\begin{longtable}{MMMM}
\caption{$584$ maximal curves over $\F_q$ of genus $g=2$.}\label{tabMax2} \\
g & q & h_W(x) & \text{Curves}  \\
\toprule
 2 & 2 & x^2+3x+1 & (15, 37) \{ 1 ; 2 ; 3 \} , \,\,  (21, 23) \{ 1 ; 2 ; 3 \} , \,\,  (93, 1) \{ 1, 2 \} , \,\,  (115, 1) \{ 1, 2 \} , \,\,  (177, 1) \{ 1 ; 2 \} , \\
  & & & (213, 1) \{ 1 ; 2 \} , \,\,  (1365, 1) \{ 1 ; 2 ; 3 ; 4 \} \\
  \midrule
  
 2 & 2^2 & x^2+5x+9 & (15, 29) \{ 1 ; 2, 3 \} , \,\,  (15, 73) \{ 1 ; 2 ; 3 \} , \,\,  (15, 91) \{ 1 ; 2 ; 3 ; 4 \} , \,\,  (21, 11) \{ 1 ; 2 \} , \,\,  (21, 43) \{ 1 ; 2 ; 3 \} , \\
  & & & (33, 7) \{ 1 ; 3 \}, \,\,  (35, 39) \{ 1 ; 2 ; 3 ; 4 \} ,  \,\, (77, 3) \{ 2 ; 3 \} , \,\, (87, 1) \{ 1 \} , \,\, (87, 5) \{ 1 ; 2 ; 3 \} ,  \\
  & & &  (111, 5) \{ 1 ; 2 ; 3 \} , \,\, (161, 1) \{ 1, 2 \} ,   \,\,  (161, 3) \{ 1 ; 2 ; 3 \} , \,\, (265, 1) \{ 1 ; 2 \} , \,\, (287, 1) \{ 1 ; 2 \} ,  \\
  & & &  (291, 1) \{ 1 ; 2 \} , \,\, (299, 1) \{ 1 ; 2 \} , \,\,  (321, 1) \{ 1 ; 2 \} , \,\, (341, 1) \{ 1 ; 2 \} , \,\, (371, 1) \{ 1 ; 2 \} , \,\,  (415, 1) \{ 1 ; 2 \} ,  \\
\bottomrule
\end{longtable}

\vspace{1cm}

\begin{longtable}{MMMM}
\caption{\mbox{$251$ maximal curves over $\F_q$ of genus $3\le g\le 4$.}}\label{tabMax34} \\
g & q & h_W(x) & \text{Curves}  \\
\toprule
 3 & 2 & x^3 + 4x^2 + 3x - 1 
 & (35, 17) \{ 1 ; 2 ; 3 \} , \,\, 
 (267, 1) \{ 1 ; 2 \} , \,\, (1995, 1) \{ 1 ; 2 ; 3 ; 4 \} \\
  \midrule 
 3 & 2^2 & x^3 + 9x^2 + 33x + 63 
 & (15, 7) \{ 2 \} , \,\, (15, 19) \{ 1, 2, 3 \} , \,\,  (15, 31) \{ 2 ; 1, 3 \} , \,\,  (15, 61) \{ 2 ; 1, 3 \} , \\
  & & & (21, 5) \{ 1 \} , \,\,  (21, 19) \{ 1, 2, 3 \} , \,\,  (21, 25) \{ 1 ; 2, 3 \} , \,\,  (21, 31) \{ 1 ; 2, 3 \} , \\
  & & & (33, 23) \{ 2 ; 1, 3 \} , \,\,  (39, 11) \{ 3 ; 1, 2 \} , \,\,  (55, 7) \{ 1 ; 2, 3 \} , \,\,  (65, 7) \{ 2 ; 1, 3 \} , \\
  & & & (95, 7) \{ 1 ; 2 ; 3 \} , \,\,  (203, 3) \{ 1 ; 2 ; 3 \} , \,\, (497, 1) \{ 1 ; 2 \} , \\
  & & & (1785, 1) \{ 1; 3; 2, 4 \} , \,\,  (1785, 1) \{ 3 ; 1, 2 ; 2, 4 \} , \,\,  (2415, 1) \{ 1 ; 2 ; 3, 4 \} \\
\bottomrule
\end{longtable}

\vspace{1cm}

\begin{longtable}{MMMM}
\caption{\mbox{$63$ maximal curves over $\F_q$ of genus $g\ge 5$.}}\label{tabMax>4} \\
g & q & h_W(x) & \text{Curves}  \\
\toprule
 5 & 2^2 & x^5 + 12x^4 + 62x^3 + 192x^2 + 436x + 864 
 & (15, 149) \{ 1 ; 3 ; 2 \} , \,\,  (21, 55) \{ 1 ; 3 ; 2, 4 \}  , \\
  & & & (51, 11) \{ 1 ; 2, 3 \}  , \,\, (185, 7) \{ 1 ; 3 ; 2 \}  \\
  \midrule 
 5 & 7^2 & x^5 + 70 x^4 + 2170 x^3 + 39200 x^2 + 458150 x + 3665984
 & (6, 115) \{ 1 ; 4 \} , \,\,  (10, 61) \{ 1, 3 \} , \,\,  (15, 8) \{ 3 \} , \\ 
  & & & (15, 16) \{ 1 ; 3 \} , \,\, (15, 41) \{ 1, 2, 3 \} , \\
  & & &  (51, 2) \{ 3 \} ,  \,\, (51, 4) \{ 1 ; 3 \} , \,\, (690, 1) \{ 2 \} \\
  \midrule 
 5 & 11^2 & x^5 + 110x^4 + 5390x^3 + 154880x^2 + 2895530x + 37097632
 &   (6, 85) \{ 4 ; 1, 2 \} , \,\, (6, 85) \{ 4 ; 1, 3 \} ,  \\
\bottomrule
\end{longtable}

\vspace{1cm} }


\section{Further tables}\label{tables_section}

\begin{longtable}{|c|c||c|c||c|c||c|c|}
\caption{The $12$ pairs $(D,N)$ in the set $\mathcal{S}$ for which we do not prove in Part~(2) of \cref{corollary: all_atkin_lehner} whether $\Aut(X_0^D(N))$ consists only of Atkin--Lehner involutions.}\label{table: unknown_automorphism_group} \\ \hline
$(D,N)$ & $g(X_0^D(N))$ & $(D,N)$ & $g(X_0^D(N))$ & $(D,N)$ & $g(X_0^D(N))$ & $(D,N)$ & $g(X_0^D(N))$\\ \hline \hline
$ (
21 , 29 )$ & $ 29 $
 & 
$ (
21 , 145 )$ & $ 177 $
 & 
$ (
69 , 17 )$ & $ 65 $
 & 
$ (
77 , 17 )$ & $ 89 $
\\ \hline
$ (
129 , 5 )$ & $ 41 $
 & 
$ (
177 , 5 )$ & $ 57 $
 & 
$ (
205 , 1 )$ & $ 13 $
 & 
$ (
437 , 1 )$ & $ 33 $
\\ \hline
$ (
497 , 1 )$ & $ 35 $
 & 
$ (
517 , 1 )$ & $ 37 $
 & 
$ (
537 , 5 )$ & $ 177 $
 & 
$ (
989 , 1 )$ & $ 77 $
\\ \hline
\end{longtable}

\begin{longtable}{|c|c|c|c|c|c|c|c|}
\caption{The $1321$ pairs $(D,N)$ in the set $\mathcal{S}$ for which we do not prove in Part~(1) of \cref{corollary: all_atkin_lehner} whether $\Aut(X_0^D(N)/W)$ consists only of Atkin--Lehner involutions for all $W \leq W_0(D,N)$.}\label{table: unknown_automorphism_group_stars} \\ \hline
$ ( 6 , 283 )$
 & 
$ ( 6 , 293 )$
 & 
$ ( 6 , 311 )$
 & 
$ ( 6 , 337 )$
 & 
$ ( 6 , 353 )$
 & 
$ ( 6 , 359 )$
 & 
$ ( 6 , 365 )$
 & 
$ ( 6 , 379 )$
\\ \hline
$ ( 6 , 383 )$
 & 
$ ( 6 , 401 )$
 & 
$ ( 6 , 413 )$
 & 
$ ( 6 , 419 )$
 & 
$ ( 6 , 421 )$
 & 
$ ( 6 , 427 )$
 & 
$ ( 6 , 443 )$
 & 
$ ( 6 , 445 )$
\\ \hline
$ ( 6 , 449 )$
 & 
$ ( 6 , 451 )$
 & 
$ ( 6 , 463 )$
 & 
$ ( 6 , 479 )$
 & 
$ ( 6 , 485 )$
 & 
$ ( 6 , 487 )$
 & 
$ ( 6 , 491 )$
 & 
$ ( 6 , 505 )$
\\ \hline
$ ( 6 , 511 )$
 & 
$ ( 6 , 515 )$
 & 
$ ( 6 , 517 )$
 & 
$ ( 6 , 533 )$
 & 
$ ( 6 , 535 )$
 & 
$ ( 6 , 551 )$
 & 
$ ( 6 , 565 )$
 & 
$ ( 6 , 581 )$
\\ \hline
$ ( 6 , 611 )$
 & 
$ ( 6 , 617 )$
 & 
$ ( 6 , 623 )$
 & 
$ ( 6 , 629 )$
 & 
$ ( 6 , 635 )$
 & 
$ ( 6 , 649 )$
 & 
$ ( 6 , 655 )$
 & 
$ ( 6 , 667 )$
\\ \hline
$ ( 6 , 679 )$
 & 
$ ( 6 , 683 )$
 & 
$ ( 6 , 685 )$
 & 
$ ( 6 , 689 )$
 & 
$ ( 6 , 695 )$
 & 
$ ( 6 , 697 )$
 & 
$ ( 6 , 707 )$
 & 
$ ( 6 , 713 )$
\\ \hline
$ ( 6 , 721 )$
 & 
$ ( 6 , 727 )$
 & 
$ ( 6 , 731 )$
 & 
$ ( 6 , 733 )$
 & 
$ ( 6 , 737 )$
 & 
$ ( 6 , 745 )$
 & 
$ ( 6 , 755 )$
 & 
$ ( 6 , 773 )$
\\
\hline
\end{longtable}

{
\begin{longtable}{|c|c|c||c|c|c||c|c|c||c|c|c|}
\caption{The $133$ pairs $(D,N)$ for which $X_0^D(N)$ is found to be geometrically tetragonal via the natural degree~$2$ map to an Atkin--Lehner quotient $X_0^D(N)/\langle w_m \rangle$ of genus $g \leq 2$, as described in the proof of Part~(1) of \cref{prop: tetragonal_sieving}.}\label{table: tetragonal_AL2} \\ \hline 
$(D,N)$ & $m$ & $g$ & $(D,N)$ & $m$ & $g$  & $(D,N)$ & $m$ & $g$  & $(D,N)$ & $m$ & $g$ \\ \hline  \hline
$ (
6 , 11 )$ & $
66 $ & $ 0 $
 & 
$ (
6 , 17 )$ & $
34 $ & $ 0 $
&
$ (
6 , 19 )$ & $
114 $ & $ 0 $
&
$ (
6 , 23 )$ & $
69 $ & $ 1 $
\\ \hline
$ (
6 , 25 )$ & $
150 $ & $ 1 $
 & 
$ (
6 , 29 )$ & $
174 $ & $ 0 $
&
$ (
6 , 31 )$ & $
186 $ & $ 0 $
 & 
$ (
6 , 37 )$ & $
222 $ & $ 0 $
\\ \hline
$ (
6 , 41 )$ & $
246 $ & $ 1 $
&
$ (
6 , 43 )$ & $
129 $ & $ 1 $
&
$ (
6 , 47 )$ & $
94 $ & $ 1 $
 & 
$ (
6 , 49 )$ & $
294 $ & $ 2 $
\\ \hline
$ (
6 , 53 )$ & $
106 $ & $ 2 $
 & 
$ (
6 , 59 )$ & $
354 $ & $ 2 $
&
$ (
6 , 61 )$ & $
366 $ & $ 2 $
&
$ (
6 , 67 )$ & $
402 $ & $ 2 $
\\ \hline
$ (
6 , 71 )$ & $
426 $ & $ 1 $
 & 
$ (
6 , 79 )$ & $
474 $ & $ 2 $
&
$ (
6 , 109 )$ & $
654 $ & $ 2 $
 & 
$ (
10 , 9 )$ & $
90 $ & $ 1 $
\\ \hline
$ (
10 , 11 )$ & $
110 $ & $ 0 $
&
$ (
10 , 13 )$ & $
65 $ & $ 0 $
&
$ (
10 , 17 )$ & $
170 $ & $ 1 $
 & 
$ (
10 , 19 )$ & $
38 $ & $ 0 $
\\
\hline
\end{longtable}
}

{
\begin{longtable}{|c|c||c|c||c|c|}
\caption{The $24$ pairs $(D,N)$ for which $X_0^D(N)$ is found to be geometrically tetragonal via the natural degree~$4$ map to an Atkin--Lehner quotient $X_0^D(N)/\langle w_{m_1}, w_{m_2} \rangle$ of genus~$0$, as described in the proof of Part~(1) of \cref{prop: tetragonal_sieving}.}\label{table: tetragonal_AL4} \\ \hline 
$(D,N)$ & $m_1, m_2$ & $(D,N)$ & $m_1, m_2$ & $(D,N)$ & $m_1, m_2$ \\ \hline \hline
$ (
6 , 35 )$ & $
6 , 70 $
 & 
$ (
6 , 55 )$ & $
10 , 165 $
 & 
$ (
6 , 65 )$ & $
2 , 195 $
\\ \hline
$ (
6 , 77 )$ & $
21 , 154 $
 & 
$ (
6 , 83 )$ & $
6 , 249 $
 & 
$ (
6 , 89 )$ & $
3 , 534 $
\\ \hline
$ (
6 , 91 )$ & $
42 , 546 $
 & 
$ (
10 , 21 )$ & $
3 , 35 $
 & 
$ (
10 , 33 )$ & $
30 , 110 $
\\ \hline
$ (
10 , 53 )$ & $
53 , 530 $
 & 
$ (
14 , 15 )$ & $
30 , 105 $
 & 
$ (
14 , 23 )$ & $
46 , 161 $
\\ \hline
$ (
21 , 13 )$ & $
39 , 91 $
 & 
$ (
22 , 19 )$ & $
38 , 209 $
 & 
$ (
34 , 7 )$ & $
14 , 17 $
\\ \hline
$ (
39 , 5 )$ & $
5 , 39 $
 & 
$ (
51 , 5 )$ & $
5 , 255 $
 & 
$ (
58 , 5 )$ & $
29 , 290 $
\\ \hline
$ (
82 , 3 )$ & $
6 , 246 $
 & 
$ (
86 , 3 )$ & $
6 , 129 $
 & 
$ (
570 , 1 )$ & $
2 , 570 $
\\ \hline
$ (
690 , 1 )$ & $
5 , 138 $
 & 
$ (
714 , 1 )$ & $
3 , 238 $
 & 
$ (
770 , 1 )$ & $
10 , 770 $
\\ \hline

\end{longtable}

\begin{longtable}{|c|c|c| |c|c|c|}
\caption{The $4$ pairs $(D,N)$ for which $X_0^D(N)$ is found to be geometrically tetragonal via the existence of a genus~$3$ quotient $X_0^D(N)/\langle w_{m_1} \rangle$ with a further quotient $X_0^D(N)/\langle w_{m_1}, w_{m_2} \rangle$ of genus~$2$, as described in the proof of Part~(1) of \cref{prop: tetragonal_sieving}.}\label{table: geom_tetragonal_by_Polizzi} \\ \hline 
$(D,N)$ & $m_1$ & $m_2$ & $(D,N)$ & $m_1$ & $m_2$  \\ \hline \hline
$(14,17)$ & $238$ & $14$ &
$(21,8)$ & $7$ & $3$ \\ \hline
$(46,7)$ & $161$ & $23$ &
$(145,1)$ & $145$ & $5$ \\ \hline 
\end{longtable}
}

{
\begin{longtable}{|c|c|c||c|c|c||c|c|c|}
\caption{The $81$ pairs $(D,N)$ for which $X_0^D(N)$ is found to be tetragonal over $\Q$ via certain Atkin--Lehner quotients with rational CM points, as described in the proof of Part~(2) of \cref{prop: tetragonal_sieving}. For each pair, we provide a Hall divisor $m \parallel DN$ used in the argument, as well as a discriminant $\Delta$ of a CM order for which $X_0^D(N)/\langle w_m \rangle$ has a rational point with CM by that order.}\label{table: tetragonal_by_CM} \\ \hline
$(D,N)$ & $m$ & $\Delta$ & $(D,N)$ & $m$ & $\Delta$ & $(D,N)$ & $m$ & $\Delta$ \\ \hline \hline
$ (
6 , 35 )$ & $
210 $ & $ -19 $
 & 
$ (
6 , 43 )$ & $
258 $ & $ -3 $
 & 
$ (
6 , 47 )$ & $
282 $ & $ -19 $
\\ \hline
$ (
6 , 53 )$ & $
318 $ & $ -4 $
 & 
$ (
6 , 59 )$ & $
354 $ & $ -43 $
 & 
$ (
6 , 61 )$ & $
366 $ & $ -3 $
\\ \hline
$ (
6 , 65 )$ & $
390 $ & $ -4 $
 & 
$ (
6 , 67 )$ & $
402 $ & $ -3 $
 & 
$ (
6 , 79 )$ & $
474 $ & $ -3 $
\\ \hline
$ (
6 , 89 )$ & $
534 $ & $ -4 $
 & 
$ (
6 , 91 )$ & $
546 $ & $ -3 $
 & 
$ (
6 , 109 )$ & $
654 $ & $ -3 $
\\ \hline
$ (
10 , 19 )$ & $
95 $ & $ -8 $
 & 
$ (
10 , 21 )$ & $
210 $ & $ -3 $
 & 
$ (
10 , 31 )$ & $
310 $ & $ -3 $
\\ \hline
$ (
10 , 37 )$ & $
370 $ & $ -3 $
 & 
$ (
10 , 53 )$ & $
530 $ & $ -43 $
 & 
$ (
14 , 11 )$ & $
77 $ & $ -8 $
\\
\hline
\end{longtable}
}

{
\begin{longtable}{|c|c|c|c||c|c|c|c||c|c|c|c|}
\caption{The $322$ pairs $(D,N)$ for which $X_0^D(N)$ is found to not be geometrically tetragonal via a Castelnuovo--Severi argument involving an Atkin--Lehner quotient map of degree~$2$, as described in the proof of Part~(3) of \cref{prop: tetragonal_sieving}. We provide an $m \parallel DN$ which suffices for the argument, along with $g = g(X_0^D(N))$ and $g_m = g(X_0^D(N))/\langle w_m \rangle$.}\label{table: non-tetragonal_CS2} \\ \hline 
$(D,N)$ & $m$ & $g$ & $g_m$ & $(D,N)$ & $m$ & $g$ & $g_m$ & $(D,N)$ & $m$ & $g$ & $g_m$ \\ \hline \hline
$ (
6 , 85 )$ & $
34 $ & $ 17 $ & $ 5 $
 & 
$ (
6 , 95 )$ & $
570 $ & $ 21 $ & $ 7 $
 & 
$ (
6 , 97 )$ & $
582 $ & $ 15 $ & $ 4 $
\\ \hline
$ (
6 , 101 )$ & $
606 $ & $ 17 $ & $ 6 $
 & 
$ (
6 , 103 )$ & $
309 $ & $ 17 $ & $ 6 $
 & 
$ (
6 , 107 )$ & $
321 $ & $ 19 $ & $ 5 $
\\ \hline
$ (
6 , 113 )$ & $
678 $ & $ 19 $ & $ 5 $
 & 
$ (
6 , 115 )$ & $
69 $ & $ 25 $ & $ 9 $
 & 
$ (
6 , 119 )$ & $
714 $ & $ 25 $ & $ 7 $
\\ \hline
$ (
6 , 121 )$ & $
726 $ & $ 23 $ & $ 7 $
 & 
$ (
6 , 125 )$ & $
750 $ & $ 25 $ & $ 8 $
 & 
$ (
6 , 127 )$ & $
381 $ & $ 21 $ & $ 6 $
\\ \hline
$ (
6 , 131 )$ & $
786 $ & $ 23 $ & $ 8 $
 & 
$ (
6 , 133 )$ & $
798 $ & $ 25 $ & $ 9 $
 & 
$ (
6 , 137 )$ & $
274 $ & $ 23 $ & $ 6 $
\\ \hline
$ (
6 , 139 )$ & $
834 $ & $ 23 $ & $ 8 $
 & 
$ (
6 , 143 )$ & $
286 $ & $ 29 $ & $ 9 $
 & 
$ (
6 , 145 )$ & $
174 $ & $ 29 $ & $ 9 $
\\
\hline
\end{longtable}
}

\normalsize

\bibliographystyle{amsalpha}
\bibliography{biblio}
\end{document}